\documentclass[11pt, letter]{article}

\usepackage[margin=1.1in]{geometry} 
\usepackage{parskip}
\usepackage{longtable,color,natbib}
\usepackage{booktabs} 
\usepackage{array} 
\usepackage{paralist} 
\usepackage{verbatim} 
\usepackage{appendix}
\usepackage{amssymb}
\usepackage{amsthm}
\usepackage{algorithm,algorithmic}
\usepackage{amsmath}
\usepackage{bbm}
\usepackage{enumitem}
\usepackage{mdframed}
\usepackage{adjustbox}
\usepackage{mathrsfs}

\definecolor{linkcolor}{rgb}{0.1,0,0.7}
\definecolor{urlcolor}{rgb}{1,0,0}
\usepackage[unicode=true,colorlinks,citecolor=linkcolor,linkcolor=linkcolor,urlcolor=blue]{hyperref}
\usepackage{graphicx} 
\usepackage{epstopdf}
\usepackage{bm}
\usepackage{cleveref}
\usepackage{pst-node}
\psset{arrowscale=2,arrows=->}

\usepackage{wrapfig}
\usepackage{multirow}
\usepackage{tabularx}
\usepackage[all,cmtip]{xy}
\usepackage{bm}

\usepackage{subcaption}
\usepackage{graphicx}

\usepackage{tikz} 
\usetikzlibrary{positioning, arrows.meta}
\usetikzlibrary{calc}

\tikzset{
  frame/.style={
    rectangle, draw,
    text width=6em, text centered,
    minimum height=4em,drop shadow,fill=white,
    rounded corners,
  },
  line/.style={
    draw, -latex',rounded corners=3mm,
  }
}

\newtheorem{Theorem}{Theorem}[section]
\newtheorem{Definition}[Theorem]{Definition}
\newtheorem{Proposition}[Theorem]{Proposition}
\newtheorem{Lemma}[Theorem]{Lemma}
\newtheorem{Corollary}[Theorem]{Corollary}
\newtheorem{Remark}[Theorem]{Remark}

\newtheorem{Assumption}[Theorem]{Assumption}

\numberwithin{equation}{section}
\usepackage{hhline}
\usepackage{verbatim}
\usepackage{eurosym}

\newcommand{\nc}{\newcommand}
\nc{\ind}{\mathds{1}}
\def \trans{^{\scriptscriptstyle{\intercal}}}

\newcommand{\R}{\mathbb{R}}

\newcommand{\E}{\mathcal{E}}
\newcommand{\F}{\mathcal{F}}

\newcommand{\bP}{\mathbb{P}}
\newcommand{\G}{\mathcal{G}}

\allowdisplaybreaks[4]

\DeclareMathOperator{\esssup}{esssup}

\def\esssup_#1{\underset{#1}{\mathrm{ess\,sup\, }}}
\def\essinf_#1{\underset{#1}{\mathrm{ess\,inf\, }}}
\def\argmax_#1{\underset{#1}{\mathrm{arg\,max\, }}}
\def\argmin_#1{\underset{#1}{\mathrm{arg\,min\, }}}

\def\b1{\bf 1}

\def \N{\mathbb{N}}
\def \R{\mathbb{R}}

\def \E{\mathbb{E}}
\def \F{\mathbb{F}}
\def \G{\mathbb{G}}

\def \P{\mathbb{P}}

\def \Ac{{\cal A}}

\def \Cc{{\cal C}}
\def \Dc{{\cal D}}
\def \Ec{{\cal E}}
\def \Fc{{\cal F}}
\def \Gc{{\cal G}}
\def \Hc{{\cal H}}
\def \Ic{{\cal I}}

\def \Lc{{\cal L}}
\def \Pc{{\cal P}}

 \def \Nc{{\cal N}}
 
\def \Sc{{\cal S}}
\def \Tc{{\cal T}}
\def \Uc{{\cal U}}
\def \Vc{{\cal V}}
\def \Wc{{\cal W}}

\def \Vc{{\cal V}}

\def \Xc{{\cal X}}

\def\eqref#1{{\rm(\ref{#1})}}
\def\beqs{\begin{eqnarray*}}
\def\enqs{\end{eqnarray*}}
\def\beq{\begin{eqnarray}}
\def\enq{\end{eqnarray}}

\begin{document}

\title{Continuous-time q-learning for mean-field control with\\ common noise, part-I: Theoretical foundations}

\author{Zhenjie Ren \thanks{Email: zhenjie.ren@univ-evry.fr, LaMME, Universit\'e \'Evry Paris-Saclay, \'Evry, France.}
\and
Xiaoli Wei \thanks{Email: tyswxl@gmail.com.}
\and
Xiang Yu \thanks{Email: xiang.yu@polyu.edu.hk, Department of Applied Mathematics, The Hong Kong Polytechnic University, Kowloon, Hong Kong.}
\and
Xun Yu Zhou \thanks{Email: xz2574@columbia.edu, Department of Industrial Engineering and Operations Research, Columbia University, New York, USA.}
}
\date{This version: April 30, 2026}

\maketitle

\begin{abstract}
This paper investigates the continuous-time counterpart of the Q-function for entropy-regularized mean-field control (MFC) with controlled common noise, coined as q-function by \cite{jiazhou2022} in the single agent's model. We first show that, under discretely sampled actions, the value function in the exploratory formulation converges to the one in the relaxed control formulation as the time grid refines. Leveraging the relaxed control formulation, we derive the exploratory Hamilton-Jacobi-Bellman (HJB) equation, in which the controlled common noise gives rise to an additional nonlinear functional of policy, rendering the policy iteration intricate. Under certain concavity condition, we establish the existence and uniqueness of the optimal one-step policy iteration via a first-order condition using the partial linear functional derivative with respect to policy. The policy improvement at each iteration is verified by relating to an entropy-regularized optimization problem over the space of policies. In the mean-field setting, we introduce the integrated q-function (Iq-function) defined on the state distribution and the policy, and it is shown that an optimal policy is identified as a two-layer fixed point to the argmax operator of the Iq-function. Finally, we provide the explicit characterization of an optimal policy as a Gaussian distribution in the general linear-quadratic (LQ) setting.

\end{abstract}

\vspace{0.15in}
\section{Introduction}
Decision making for a large population system with interacting agents in a competitive or cooperative manner has wide applications across finance, systemic risk control, epidemic control, robot swarms,  traffic management, among others. The main challenge in the large system is to understand the coupled influence of decision making on the behavior of all agents. To overcome this complexity and the curse of dimensionality in numerical implementations, the mean-field approximation of the population's state, proposed independently by \cite{LL} and \cite{Huangetal2006}, has been extensively studied over the past decades. The mean-field formulation allows one to study the weak interaction between one representative agent and the population rather than the coupled interactions between agents. Mean-field game (MFG) and mean-field control (MFC) problems have been proposed and developed to model the competitive and cooperative interactions, respectively, see \cite{Car13} for the discussion on their relationship and distinctions. See also \cite{CarD,CarD2} for an extensive overview of existing studies in these two types of problems.

In the present paper, we are interested in MFC problems where a social planner coordinates all agents to optimize the aggregated reward function that leads to the social optimum of the population. 
More importantly, we incorporate the Brownian common noise in the mean-field system that affects the entire population simultaneously. The study of common noise in large population system has recently attracted a lot of attention and spurred new advances in mean-field theories, which can effectively describe the exogenous random risk acting towards the whole system. Unlike the idiosyncratic noise that only affects the specific state dynamics of an individual agent, the presence of common noise leads to the mean-field interaction via the conditional population distribution given the common noise. Thereby, in MFC problems, we encounter the conditional population distribution as a measure-valued process that calls for It\^{o}'s formula along conditional measure flows and the stochastic Fokker-Planck equation, giving rise to many new technical challenges, see \cite{G16}, \cite{phamwei2017}, \cite{BCL2021}, \cite{DPT2022}, \cite{mottepham2022}, \cite{CQB23}, \cite{Zhou2024} for some recent developments in MFC problems with common noise.

Conventional methods to solve the classical stochastic control and MFC problems typically assume the full knowledge or precise estimations of the underlying model. However, in reality, the agent or the social planner may only have little or no information of the environment. The limited information of unknown environment may cause huge errors or inefficiency in implementing the theoretical solutions. This motivated an upsurge of interests in studying the reinforcement learning (RL) algorithms from the classical single agent's model to the large stochastic systems. Based on the trial-and-error interactions with the unknown environment, the decision maker can gradually learn to select best actions from the procedure of exploration and exploitation. Albeit the substantial success of RL algorithms in wide applications, the theoretical study of RL has been predominantly limited to discrete time models. However, many real-world applications, particularly in finance and engineering, evolve continuously through time. Recently, for continuous-time stochastic control problems by a single agent, \cite{Wangetal2021}, \cite{JZ22a, JZ22b, jiazhou2022} have laid the theoretical foundation for RL with entropy regularization with continuous state space and action space. To learn an optimal policy in a continuous-time setting requires the shift from the discrete Bellman equation in conventional RL studies to its continuous-time counterpart, the Hamilton-Jacobi-Bellman (HJB) equation. In particular, \cite{Wangetal2021} studied the optimal policy  in an entropy-regularized exploratory RL framework for diffusion processes. \cite{JZ22a} examined the policy evaluation problem by establishing a martingale condition of the value function. \cite{JZ22b} studied the policy gradient algorithm by connecting it to the martingale approach in \cite{JZ22a}. \cite{jiazhou2022} developed the continuous-time q-learning theory by introducing the q-function as the first order time derivative of the advantage function and establishing a joint martingale characterization of the q-function and the value function. This continuous-time entropy-regularized RL method has been rapidly generalized in various context of single agent's control problems. For example, \cite{W23} proposed an actor-critic RL algorithm for optimal execution in continuous-time Almgren-Chriss model;  \cite{Han23} considered the Choquet entropy regularization for the RL exploration and explored the distribution of the optimal policy in the LQ framework; \cite{Dong2022} examined the entropy-regularized RL method for optimal stopping problems; \cite{DDJ} studied the policy iteration algorithms to learn the time-consistent equilibrium policy for mean-variance portfolio optimization problems; \cite{BHY23} generalized the q-learning algorithm for reflected diffusion dynamics and applied the algorithm in solving the optimal tracking portfolio problem; \cite{DDJZ2023} considered the recursive entropy regularization and developed the RL algorithm to learn Merton's optimal strategy in an incomplete market model; \cite{HLYZ25} examined the continuous-time reinforcement learning approach for optimal switching over multiple regimes; \cite{J26} investigated the continuous-time risk-sensitive reinforcement learning using the quadratic variation penalty. 

Comparing with the large volume of studies in continuous-time RL for single agent's control problems, the investigations of continuous-time RL for MFG and MFC are relatively underdeveloped. In the LQ-MFG, \cite{GuoXZ} examined the theoretical justification that entropy regularization helps stabilizing and accelerating the convergence to the Nash equilibrium. \cite{FGLPS23} generalized the policy gradient algorithm in \cite{JZ22b} to continuous time MFC problems and devised actor-critic algorithms based on a gradient expectation representation of the value function. \cite{Liangetal2024} similarly generalized the actor-critic policy gradient algorithms to continuous-time MFG problems together with fictitious play to update the population distribution. As a first attempt to generalize the continuous-time q-learning to MFC problems, \cite{weiyu2023} recently proposed the integrated q-function (Iq-function) and the essential q-function, which facilitate the design of q-learning algorithms for MFC without common noise from the social planner's perspective. \cite{weiyuyuan2023} further discussed the proper Iq-function in decoupled form and proposed unified q-learning algorithms for both MFG and MFC problems without common noise from the representative agent's perspective.

Similar to \cite{weiyu2023} in the setting without common noise, it is assumed in the present paper that the social planner is responsible to learn a policy that maximizes the collective profit of the population. That is, the social planner assigns randomized policies to agents who interact with the environment based on their own current states and population's state distribution. Based on agents' interactions, the social planner collects the population's distribution and agent's individual rewards to generate the population's aggregated reward and iterate the policies accordingly to learn the optimal policy. A typical example of such setting with learning social planner could be the centralized traffic management system while the population of agents could refer to the drivers and public transit users. The unknown environment is the city's road network and the learning task of the social planner is to dynamically adjust the timing and phasing of all interconnected traffic lights; set speed limits; or implement dynamic congestion pricing.  See also Appendix \ref{appendix:N-player} for a brief description of the interactions of $N$ players in the RL setting that motivates the RL formulation of the MFC problem.

\subsection{Our contributions}
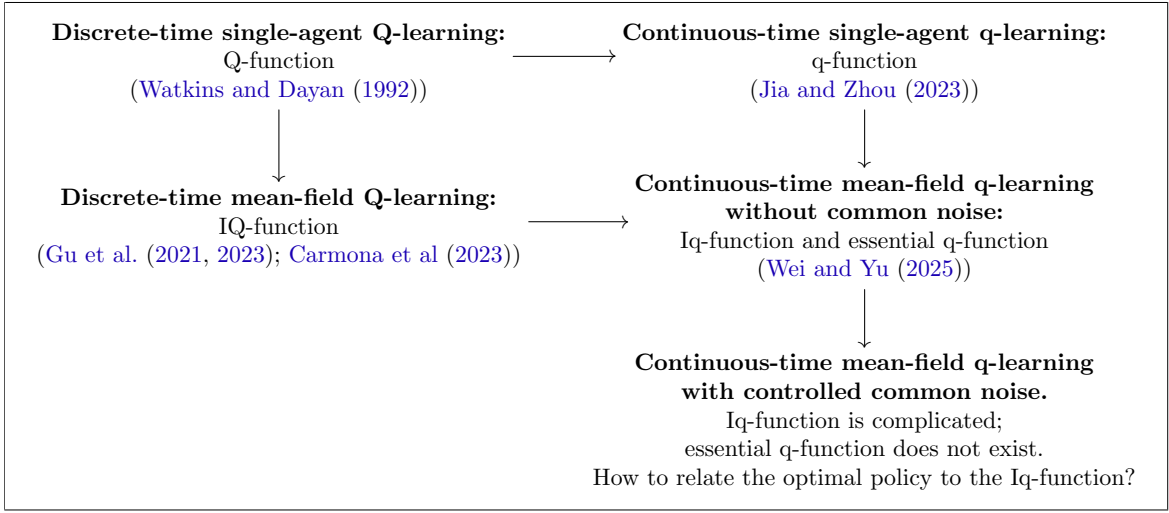
\begin{figure}[h]
\centering
\adjustbox{varwidth=\linewidth,scale=0.9}{
\begin{mdframed}
\xymatrix{
\txt{\small{\textbf{Discrete-time single-agent Q-learning:}}\\\small{Q-function} \\ \small{(\cite{W92})}}\ar@[black][d] \ar@[black][r] & \txt{\small{\textbf{Continuous-time single-agent q-learning:}}\\\small{q-function} \\\small{(\cite{jiazhou2022})}} \ar@[black][d]\\
\txt{\small{\textbf{Discrete-time mean-field Q-learning:}}\\\small{{IQ-function}}\\ \small{(\cite{GGWX21, GGWX23, CLT23})} } \ar@[black][r]   & \txt{\small{{\textbf{Continuous-time mean-field q-learning}}}\\\small{\textbf{without common noise:}}\\\small{Iq-function and essential q-function}\\ \small{(\cite{weiyu2023})}} \ar@[black][d]\\
\ \ & \txt{\small{\textbf{Continuous-time mean-field q-learning}}\\\small{\textbf{with controlled common noise.}}\\\small{Iq-function is complicated;}\\\small{essential q-function does not exist.}\\\small{How to relate the optimal policy to the Iq-function?}}}
\end{mdframed}
\caption{\centering Conceptual relationship to the literature}\label{F1}}
\end{figure}

The goal of the present paper is to investigate the correct form of Iq-function and develop some q-learning theories in the presence of controlled common noise, see the illustration of our research motivation in Figure \ref{F1}. The design of q-learning algorithms as well as establishing some supporting theories, such as the martingale condition of the Iq-function and value function, are studied in our accompany paper \cite{RWYZ25}. 

The controlled common noise creates many new difficulties in the learning procedure. First, identifying an appropriate relaxed control formulation suitable for theoretical analysis in the presence of controlled common noise remains an open problem. In particular, we need to introduce additional Brownian motions in the relaxed control formulation (see \eqref{equ:common-noise-average}) so that the joint law of the state and the common noise $\Lc(X, B)$ coincides with that of the exploratory formulation. This ensures the consistency of the conditional state distributions $\Lc(X|B)$ in two formulations. Second, It\^{o}'s formula on the flow of conditional probability measure yields a more sophisticated exploratory HJB equation with an additional nonlinear functional of policy, see the PDE \eqref{equ:dynamic-programming-equation}. As a result, the one-step policy iteration by the first-order condition no longer admits an explicit form, in sharp contrast to the explicit Gibbs measure iteration rule in \cite{jiazhou2022, weiyu2023}. This nonlinear functional of policy pins down how common noise may fundamentally complicate the iteration in learning. In fact, whether the policy improvement under this implicit iteration operator holds or not deserves some careful investigations. By some heuristic computations, we shall expect to see that the definition of the Iq-function in the present paper also involves this nonlinear functional of policy, see Definition \ref{def:coupled-q-function}. As the optimal policy is no longer explicitly related to the Hamiltonian in \eqref{operatorH}, how to learn the optimal policy via the learnt Iq-function is a key issue to address in devising some RL algorithms. We summarize the main contributions of the present paper as follows:
\begin{itemize}
\item[(i)] To cope with the additional nonlinear functional of policy in the exploratory HJB equation caused by controlled common noise, we establish a rigorous characterization of the unique optimal policy as a two-layer fixed point to an iteration operator using the notion of the partial linear functional derivative  with respect to the policy, see Theorem \ref{lemma:first-order} and Corollary \ref{cor:policy-improvement}. Moreover, the policy iteration operator is implicit as for a given policy, to exercise the optimal one-step policy iteration, one needs to solve a fixed point problem stemming from the first order condition equation in \eqref{equ:first-order}, which differs significantly from previous results in \cite{jiazhou2022, weiyu2023}. Thanks to some proper concavity condition, we can also prove the policy improvement result for this implicit policy iteration operator, see Theorem \ref{thm:policy_improvement}.

\item[(ii)] Using the definition of IQ-function in the discrete-time MFC model and It\^{o}'s lemma along the flow of conditional probability measures, we derive the proper definition of the Iq-function in Definition \ref{def:coupled-q-function}, which also carries the nonlinear functional of policy. It then follows from the previous first-order condition that the optimal policy is related to a two-layer fixed point of the argmax operator of the Iq-function. Equivalently, we show that it also corresponds to a two-layer fixed point of the operator in the Gibbs measure form using the partial linear functional derivative of the unregularized Iq-function with respect to the policy (see the expression \eqref{twofix} in Corollary \ref{corollary:optimal-policy}).

\item[(iii)] In the general LQ setting, we pioneer the explicit characterization of the optimal policy as a two-layer fixed point of the implicit policy iteration operator, which is shown to be a Gaussian distribution; see Theorem \ref{thm:LQ} and the optimal policy in \eqref{LQ-optimal-policy}. This justifies the ad hoc choice of Gaussian policy in learning tasks even in the presence of common noise.

\end{itemize}

The remainder of the paper is organized as follows. Section \ref{sec:formulation} reviews the classical MFC problem with common noise and introduces its relaxed control and exploratory formulations. Section \ref{sec:policy-improvement} gives the characterization of an optimal policy as the two-layer fixed point to an implicit iteration operator and establishes the policy improvement result. Section \ref{sec:q-function} proposes the proper definition of the continuous-time Iq-function and characterizes an optimal policy as the two-layer fixed point to the argmax operator of the Iq-function. Section \ref{sec:LQ} investigates the general LQ-MFC problem with controlled common noise and establishes the explicit characterization of the two-layer fixed point as a Gaussian policy. Finally, Appendix \ref{appendix:N-player} briefly discusses the continuous-time RL for the cooperative $N$-player game and Appendix \ref{sec:exploratory-average} presents the heuristic derivation of the relaxed control formulation.

\ \\
\paragraph{Notations} Given two Polish spaces $(S, \Sc)$ and $(T, \Tc)$, for any $p >0$, we denote by $\Pc_p(S)$ the space of all probability measures with finite $p$-th moment on $S$ and equip $\Pc_p(S)$ with the $p$-Wasserstein metric $\Wc_p$ defined by
\begin{align*}
\Wc_p(\mu, \nu) = \inf \Big\{\Big(\int_{S \times S} |x-y|^p \gamma(dx, dy)\Big)^{1/p}: \gamma \in \Pc_p(S \times S)\; \mbox {has marginals} \; \mu \; \mbox{and} \; \nu\Big\}.
\end{align*}
For $\mu \in \Pc_p(\R^d)$, denote $\|\mu\|_p: = \big(\int_{\R^d} |x|^p \mu(dx)\big)^{1/p}$. Let $\Pc_{ac}(S)$ be the space of probability measures on $S$ that are absolutely continuous with respect to the Lebesgue measure. We denote by $\Pc(T|S)$ (resp. $\Pc_{ac}(T|S)$) the space of all probability kernels from $S$ to $\Pc(T)$ (resp. $\Pc_{ac}(T)$). $L^2(\Omega, \Fc, \P; \R^d)$ denotes the space of all $\Fc$-adapted $\R^d$-valued square-integrable random variables on the probability space $(\Omega, \Fc, \P)$. For any measurable function $g:S\to \R^k$, we denote $\int_{S} g(x)\mu(dx)$ by $\langle g, \mu\rangle$. For a functional $F: \Pc_2(S) \to \R$, $\partial_\mu F(\mu)(x)$, $\partial_x\partial_\mu F(\mu)(x)$ and $\partial_\mu^2 F(\mu)(x, x')$ stand for the $L$ derivative in $\mu$, the mixed second-order derivative with respect to $\mu$ and $x$, and the second-order derivative in measure $\mu$, respectively (see  Definition 5.22 in \cite{CarD}). We use $\Nc(\mu, \Sigma)$ and $\Uc([p, q])$ for a Gaussian distribution with mean $\mu$ and covariance $\Sigma$ and a uniform distribution on $[p, q]$, respectively.

\section{Problem Formulation}\label{sec:formulation}
\subsection{Classical MFC problem with common noise}\label{sec:classical-mfc}
Let $(\Omega, \Fc, \P)$ be a complete probability space with a product structure $(\Omega^0 \times \Omega^1, \Fc^0 \otimes \Fc^1, \P^0 \otimes \P^1)$, where $(\Omega^1, \Fc^1, \P^1)$ supports a  $m$-dimensional Brownian motion $W=(W_s)_{s \in [0, T]}$ and $(\Omega^0, \Fc^0, \P^0)$ supports a $n$-dimensional Brownian motion $B = (B_s)_{s \in [0, T]}$ with $B$ serving as common noise. We consider two filtrations $\F^{W, B} = (\Fc_t^{W, B})_{0 \leq t \leq T}$ and $\G = (\Gc_t)_{0 \leq t \leq T}$ defined by $\Fc_t^{W,B} := \sigma(W_s, B_s: s \in [0, t])$ and $\Gc_t := \sigma(B_s: s \in [0, t])$, respectively. It is assumed that there exists a sub-algebra $\Hc$ of $\Fc^1$ such that $\Hc$ is independent of $\F^{W, B}$ and it is ``rich enough" in the sense that for any $\mu \in \Pc_2(\R^d)$, there exists an $\Hc$-measurable random variable $\xi$ on $(\Omega^1, \P^1)$ such that $\P^1 \circ \xi^{-1} = \mu$.
Let $\F = (\Fc_s)_{s \geq 0}$ be the filtration $\Fc_s = \Fc_s^{W, B} \vee \Hc$. 

We consider a MFC problem by the social planner, for which the representative agent's state process $\{X_s\}_{s \geq t}$, taking values in $\R^d$, is described by the controlled conditional McKean-Vlasov SDE that
\begin{align}
dX_s &= b(s, X_s, \mu_s, a_s) ds  + \sigma(s, X_s, \mu_s, a_s) dW_s + \sigma_o(s, X_s, \mu_s, a_s) dB_s, \label{Xdy}
\end{align}
where $\xi \in L^2(\Omega^1, \Hc, \P^1; \R^d)$ such that $\Lc(\xi|\Gc_t) = \mu  \in \Pc_2(\R^d)$, and $b$, $\sigma$ and $\sigma_o$ are measurable functions valued in $\R^d$, $\R^{d \times m}$ and $\R^{d \times n}$. Recall that, $W$ stands for the idiosyncratic noise for each representative agent and $B$ plays the role of common noise affecting the entire population. The conditional law $\mu_s: = \Lc({X_s}|\Gc_s)$ denotes the conditional regular probability distribution of $X_s$ given $\Gc_s$ that $\mu_s(\omega^0) = \P^1 \circ X_s^{-1}(\omega^0, \cdot)$ for every $\omega^0 \in  \Omega^0$.

The goal of the social planner in the MFC problem is to find an optimal $\F$-progressively measurable control $\{a_s\}_{t \leq s \leq T}$ valued in the space $\Ac$, a closed subset of  $\R^m$, which maximizes the expected discounted total reward that
\begin{align}\label{equ:classical-value}
\E\left[\int_t^T e^{-\beta(s-t)} r(s, X_s, \mu_s, a_s)ds + e^{-\beta (T -t)}g( X_T, \mu_T)\Big| X_t = \xi\right].
\end{align}

To ensure the wellposedness of \eqref{Xdy}-\eqref{equ:classical-value}, we impose the following assumptions.

\begin{Assumption} \label{ass:b-sigma-sigmao}The following conditions for the state dynamics and reward functions hold:
\begin{enumerate}[label=\upshape(\roman*)]
\item $b$, $\sigma$, $\sigma_o$, $r$ are jointly continuous in $(t, x, \mu, a) \in [0, T] \times \R^d \times \Pc_2(\R^d) \times \Ac$, and $g$ is jointly continuous in $(x, \mu) \in \R^d \times \Pc_2(\R^d)$;
\item There exists a constant $C >0$ such that for $f \in \{b, \sigma, \sigma_o\}$, and all $t, t' \in [0, T]$, $x, x' \in \R^d$, $\mu, \mu' \in \Pc_2(\R^d)$, $a \in \Ac$, it holds that
    \begin{align*}
    &|f(t, x, \mu, a) -f(t', x', \mu', a)| \leq C\big(|t- t'| +|x - x'| + \Wc_2(\mu, \mu')\big),\\
    &|r(t, x, \mu, a) - r(t', x', \mu, a)| \leq C\big(1 + |x| + |x'| + \|\mu\|_2 + \|\mu'\|_2\big) \big(|t- t'| + |x - x'| + \Wc_2(\mu, \mu')\big).
    \end{align*}
\item There exists some constant $C > 0$ such that for each $(t, x, \mu, a) \in [0, T] \times \R^d \times \Pc_2(\R^d) \times \Ac$, it holds that
\begin{align*}
|b(t, x, \mu, a)|  & \leq C\Big(1 + |x| + \|\mu\|_2 + |a|\Big),\\
\big|(\sigma\sigma\trans + \sigma_o\sigma_o\trans)(t, x, \mu, a)\big| &\leq  C\big(1 + |x|^2 + \|\mu\|_2^2 + |a|^2\big),\\
|r(t, x, \mu, a)| + |g(x, \mu)| & \leq C\big(1 + |x|^2 + \|\mu\|_2^2 + |a|^2\big).
\end{align*}
\end{enumerate}
\end{Assumption}

\subsection{Relaxed control formulation of MFC with common noise}
It is assumed that the mean-field model is unknown, i.e., we do not have the exact information of the model coefficients $b$, $\sigma$ and $\sigma_o$ in state dynamics \eqref{Xdy} and the reward function $r$ in \eqref{equ:classical-value}. To determine the optimal control in an unknown model, we choose to apply the RL approach based on the principle of trial-and-error recovery. To capture the exploration step in RL, we randomize the action and consider its distribution as a relaxed control. Therefore, it is necessary to extend the original filtered probability space $(\Omega, \Fc, \F, \P)$ for the purpose of the action randomization and consider another atomless probability space $(\Omega^2, \Fc^2, \P^2)$ that supports an $\Fc^2$-measurable random variables $U_0$ with uniform distribution on $[0, 1]$. By standard separation of the decimals of $U_0$ (c.f. Lemma 2.21 in \cite{Kall2002}), there exists an i.i.d. sequence of $\mathcal{F}^2$-adapted uniform random variables $(U_i)_{i \in \mathbb{N}}$, independent of $\xi$, $W$ and $B$. Denote $(\Omega^e, \Fc^e, \F^e, \P^e): = (\Omega \times \Omega^2, \Fc \otimes \Fc^2, \P \otimes \P^2)$ where $\F^e = (\Fc^e_t)_{0 \leq t\leq T}$ and $\Fc_t^e = \Fc_t \vee \sigma(U_i, t_i \leq t)$. Accordingly, for an element $\omega^e \in \Omega^e$, we write it as $\omega^e=(\omega, \omega^2) \in \Omega \times \Omega^2$, and we extend canonically $W$ and $B$ on $\Omega$ by setting $W(\omega^e): = W(\omega)$ and $B(\omega^e): = B(\omega)$. Any random variable on $\Omega$ can be extended similarly to that on $\Omega^e$. $\E^e$ stands for the expectation under $\P^e$.

We first introduce the relaxed control formulation for theoretical analysis.  Let $\Pi$ stand for the set of admissible policies satisfying the following definition.

\begin{Definition} \label{def:pi}A policy ${\bm \pi}$ is called admissible if
\begin{enumerate}[label=\upshape(\roman*)]
\item ${\bm \pi}(\cdot|t, x, \mu) \in \Pc_{ac}(\Ac)$, $\mbox{supp} {\bm \pi}(\cdot|t, x, \mu) = \Ac$ for every $(t, x, \mu) \in [0, T] \times \R^d \times \Pc_2(\R^d)$, and ${\bm \pi}$ is jointly measurable with respect to $(t, x, \mu) \in [0, T] \times \R^d \times \Pc_2(\R^d)$.
\item There exits a constant $C>0$ independent of $(t, a)$ such that for any $\bm \pi \in \Pi$ and any $x, x' \in \R^d$ and $\mu, \mu' \in \Pc_2(\R^d)$
\begin{align*}
\int_{\Ac} |{\bm \pi}(a|t, x, \mu) - {\bm \pi}(a|t, x', \mu')|da \leq C\big(|x - x'| + \Wc_2(\mu, \mu')\big).
\end{align*}
\item There exists some $C > 0$  and $\delta >0$ independent of $(t, a)$ such that for every ${\bm \pi} \in \Pi$
\begin{align*}
\int_{\Ac} |a|^{2 + \delta} {\bm \pi}(a|t, x, \mu)da \leq C(1 + |x|^{2} + \|\mu\|_{2}^{2}).
\end{align*}
\item  There exists a constant $C> 0$ such that for any $\bm \pi \in \Pi$
\begin{align*}
\big|E_{\bm \pi}(t, x, \mu)\big| \leq& C\big( 1 + |x|^2 + \|\mu\|_2^2\big)\\
\big| E_{\bm \pi}(t, x, \mu) - E_{\bm \pi}(t', x', \mu')\big| \leq& C\big(|t-t'| + |x- x'| + \Wc_2(\mu, \mu')\big),
\end{align*}
where the Shannon entropy $E_{\bm \pi}$ is defined by
\begin{align}\label{Shannon}
E_{\bm \pi}(t, x, \mu) = - \int_{\Ac} \log {\bm \pi}(a|t, x, \mu) {\bm \pi}(a|t, x, \mu)da,
\end{align}
\end{enumerate}
\end{Definition}
For $f \in \{b, \sigma, \sigma_o, r\}$, we denote by $f_{\bm \pi}$ the mean of $f$ with respect to ${\bm \pi} \in \Pi$ that
\begin{align}\label{mean-w.r.t-pi}
{f}_{\bm \pi}(t, x, \mu) &:= \int_{\Ac} f(t, x, \mu, a) {\bm \pi}(a|t, x, \mu) da, 
\end{align}
and we denote by ${\rm cov}_{\bm \pi}(f)$ and ${\rm std}_{\bm \pi}(f)$ the covariance and standard deviation of $f$ with respect to ${\bm \pi} \in \Pi$ that
\begin{align}\label{variance-w.r.t-pi}
{\rm cov}_{\bm \pi}(f)(t, x, \mu)  := &\int_{\Ac} ff\trans (t, x, \mu, a){\bm \pi}(a|t, x, \mu)da - f_{{\bm \pi}} f_{{\bm \pi}}\trans(t, x, \mu),\\
{\rm std}_{\bm \pi}(f) := & {\rm cov}_{\bm \pi}(f)^{1/2}. \label{std-w.r.t-pi}
\end{align}
To save notations, we only present the relaxed control formulation for $d=1$ here (see \eqref{equ:multi_exploratory_HJB} in Appendix \ref{sec:exploratory-average} for the case $d > 1$).
\begin{align}
d X_s^{\bm \pi} &= b_{\bm \pi}(s, X_s^{\bm \pi}, \mu_s^{\bm \pi})ds +  \sigma_{\bm \pi}(s, X_s^{\bm \pi}, \mu_s^{\bm \pi})dW_s + \sigma_{o, {\bm \pi}}(s,  X_s^{\bm \pi}, \mu_s^{\bm \pi})dB_t \nonumber\\
 &\;\;\; + {\rm std}_{{\bm \pi}}(\sigma)(s, X_s^{\bm \pi}, \mu_s^{\bm \pi})d{\overline W}_s + {\rm std}_{\bm \pi}(\sigma_o)(s, X_s^{\bm \pi}, \mu_s^{\bm \pi})d {\bar B}_s , \label{equ:common-noise-average}
\end{align}
where $\mu_s^{\bm \pi} = \Lc(X_s^{\bm \pi}|\Gc_s)$, and $\overline W$ and $\bar B$ are {\it extra} one-dimensional Brownian motions independent of $W$ and $B$. Under Assumption \ref{ass:b-sigma-sigmao} and Definition \ref{def:pi}, all coefficients in \eqref{equ:common-noise-average} are Lipschitz continuous in arguments $x$ and $\mu$. Therefore the SDE \eqref{equ:common-noise-average} admits a unique strong solution, see Theorem 5.1.1 in \cite{Stroock1997} or Appendix A in \cite{DPT2022}. 

\begin{Remark}
Due to the controlled common noise, the above relaxed control formulation is different from that in \cite{jiazhou2022, weiyu2023}. We need to introduce auxiliary Brownian motions arising from the action randomiziation in the relaxed formulation; see more details in Appendix \ref{sec:exploratory-average}.
\end{Remark}

To encourage the exploration in the continuous-time framework, we adopt the Shannon differential entropy as in \cite{Wangetal2021} 
and we consider the value function in the relaxed control formulation by
\begin{align}\label{equ:exploratory_average_value_function}
\tilde J(t, \mu; {\bm \pi}) 
& = \E^e\Big[\int_t^T e^{-\beta(s -t)}\Big(\hat r_{\bm \pi} (s, \mu_s^{\bm \pi}) + \gamma \mathcal{E}(s, \mu_s^{\bm \pi}, {\bm \pi})\Big)ds + e^{-\beta (T - t)}\hat g(\mu_T^{\bm \pi})\Big],
\end{align}
where we denote
\begin{align*}
\mathcal{E}(t,\mu, {\bm \pi}) :&=\int_{\R^d} E_{\bm \pi}(t, x, \mu)\mu(dx),\; \hat g(t, \mu) : = \int_{\R^d} g(x, \mu)\mu(dx), \\
 \hat r(t, \mu, {\bm \pi}): &= \int_{\R^d \times \Ac} r(t, x, \mu, a){\bm \pi}(a|t, x, \mu)da\mu(dx).
\end{align*}
The optimal value function is given by
\begin{align}\label{average:optimal-value}
\tilde J^*(t, \mu) = \sup_{{\bm \pi} \in \Pi} \tilde J(t, \mu; {\bm \pi}).
\end{align}

\subsection{Exploratory formulation under discretely sampled actions}\label{sec:exploratory-form}

Although the relaxed control formulation is convenient for the theoretical analysis such as deriving the HJB equation, it cannnot be directly observed in practice. As a result, to develop some implementable algorithms, we need to consider the continuous-time exploratory formulation with sampling. However, as discussed in recent studies among \cite{Szpruch2024, BenderThuan24, jiaetal25, CL25}, continuously sampling procedure requires continuum independent draws from a distribution and may cause some measure-theoretical issues. As a remedy, we next proceed to consider the discretely sampled action processes as in \cite{jiaetal25} in our mean-field setting.

Fix ${\bm \pi} \in \Pi$ and an initial pair $(t, \xi) \in [0, T] \times L^2(\Omega^1, \Hc, \P^1; \R^d)$, we then consider the {\it exploratory discretely sampled state process} for all $0 \leq i \leq n-1$ and $s \in [s_i, s_{i+1}]$ that
\begin{align}\label{equ:exploratory_SDE}
X_s^{\Dc, \bm \pi} &= X_{s_i}^{\Dc, \bm \pi} + \int_{s_i}^s b(u, X_u^{\Dc, \bm \pi}, \mu_u^{\Dc, \bm \pi},  a_{s_i}^{\bm \pi}) du + \int_{s_i}^s \sigma(u, X_u^{\Dc, \bm \pi},  \mu_u^{\Dc, \bm \pi}, a_{s_i}^{\bm \pi}) dW_u\\
& \;\;\;\;\;+ \int_{s_i}^s\sigma_o(u, X_u^{\Dc, \bm \pi},  \mu_u^{\Dc, \bm \pi}, a_{s_i}^{\bm \pi}) dB_u, \nonumber
\end{align}
where $\mu_s^{\Dc, \bm \pi} = \Lc(X_s^{\Dc, \bm \pi}|\Gc_s)$, and $a_{s_i}^{\bm \pi} = \phi_{\bm \pi}(s_i, X_{s_i}^{\Dc, \bm \pi}, \mu_{s_i}^{\Dc, \bm \pi}, U_i)$ $\sim {\bm \pi}(\cdot|s_i, X_{s_i}^{\Dc, \bm \pi}, \mu_{s_i}^{\Dc, \bm \pi})$ for some measurable function $\phi_{\bm \pi}: [0, T] \times \R^n \times \Pc_2(\R^n) \times [0, 1] \to \Ac$ stands for the sampled actions. Under Assumption \ref{ass:b-sigma-sigmao}, the SDE \eqref{equ:exploratory_SDE} is well-posed. We may also rewrite \eqref{equ:exploratory_SDE} as
\begin{align}\label{equ:exploratory_SDE1}
dX_s^{\Dc, \bm \pi} &= b(s, X_s^{\Dc, \bm \pi}, \mu_s^{\Dc, \bm \pi},  a_{\delta(s)}^{\Dc, \bm \pi}) ds + \sigma(s, X_s^{\Dc, \bm \pi}, \mu_s^{\Dc, \bm \pi}, a_{\delta(s)}^{\bm \pi}) dW_s \\
& \;\;\;\;\;+\sigma_o(s, X_s^{\Dc, \bm \pi}, \mu_s^{\Dc, \bm \pi}, a_{\delta(s)}^{\Dc, \bm \pi}) dB_s, \nonumber
\end{align}
where $\delta(s): = s_i$ for $s \in [s_i, s_{i+1})$.

The procedure of RL for MFC is then described as follows. On an arbitrary time gird $\Dc = \{t=s_0 < s_1 < \ldots < s_n = T\}$,  the representative agent at the current state $X_s^{\Dc, \bm \pi}$  observes the current population's conditional state distribution $\mu_s^{\Dc, \bm \pi}$ and takes the sequence of actions $(a_{s_i})_{s_i \in \Dc}$ {\it only at the time grid} in $\Dc$ according to a policy ${\bm \pi} \in \Pi$ assigned by the social planner. The representative agent will receive a stream of running individual rewards and her state evolves according to the sampled SDE in \eqref{equ:exploratory_SDE}. Based on the representative agent's interactions with the unknown environment, the social planner coordinates the population by assigning policies to the representative agent and collecting the population's conditional state distribution and the aggregated reward.

The value function under the time grid $\Dc$ and the policy ${\bm \pi}$ is given by
\begin{align}\label{equ:coupled_value_function}
J^{\Dc}(t, \xi; {\bm \pi})& = \E^e\biggl[\int_t^T e^{-\beta(s-t)} \big(r(s, X_s^{\Dc, \bm \pi}, \mu_s^{\Dc, \bm \pi}, a_{\delta(s)}^{\Dc, \bm \pi})
+ \gamma E_{\bm \pi}(\delta(s), X_{\delta(s)}^{\Dc, \bm \pi}, \mu_{\delta(s)}^{\Dc, \bm \pi}) \big)ds\nonumber\\
&\;\;\;\;\;\; + e^{-\beta (T -t)}g( X_T^{\Dc, \bm \pi}, \mu_T^{\Dc, \bm \pi})\Big| X_t^{\Dc, \bm \pi}= \xi \biggl].
\end{align}
In view of the arbitrariness of $\Dc$, we define the value function by taking the limit over all time grids
$J(t, \xi; {\bm \pi}) = \lim_{|\Dc| \to 0} J^{\Dc}(t, \xi; {\bm \pi})$,
where $|\Dc| := \max_{0 \leq i \leq n-1} |s_{i+1} - s_i|$.
The goal of the social planner is to maximize $J(t, \xi; {\bm \pi})$ that
\begin{align}\label{equ:optimal_value_function}
J^*(t, \xi) = \sup_{{\bm \pi} \in \Pi} J(t, \xi; {\bm \pi}).
\end{align}

\section{Policy Iteration and Policy Improvement}\label{sec:policy-improvement}

\subsection{Relationship between two formulations}

In this subsection, we first investigate the connection between the relaxed control formulation and the limit of the exploratory formulation using discretely sampled actions.

We first recall the following definition that is frequently used in the rest of the paper.
\begin{Definition}
We say that $V:[0, T] \times \Pc_2(\R^d) \to \R$ belongs to $\Cc^{1, 2}([0, T] \times \Pc_2(\R^d))$ if
\begin{itemize}
\item $\frac{\partial V}{\partial t}(t, \mu)$ exists and is jointly continuous in $t, \mu$;
\item $\partial_\mu V(t, \mu)(x)$, $\partial_x\partial_\mu V(t, \mu)(x)$ and $\partial_\mu^2 V(t, \mu)(x, x')$ exist for any $(t, \mu, x, x') \in [0, T] \times \Pc_2(\R^d) \times \R^d \times \R^d$;
\item $\partial_\mu V(t, \mu)(x)$, $\partial_x\partial_\mu V(t, \mu)(x)$ and $\partial_\mu^2 V(t, \mu)(x, x')$  are Lipschitz continuous with respect to all entries and satisfy that, for any $(t, \mu, x, x') \in [0, T] \times \Pc_2(\R^d) \times \R^d \times \R^d$,
    \begin{align*}
    |\partial_\mu V(t, \mu)(x)| \leq C(1 + |x| + \|\mu\|_2), \; |\partial_x\partial_\mu V(t, \mu)(x)| + |\partial_\mu^2 V(t, \mu)(x, x')| \leq C,
    \end{align*}
    for some constant $C > 0$.
\end{itemize}
\end{Definition}

\begin{Assumption}\label{ass:approximation}
$b, \sigma$, and $\sigma_0$ are sufficiently regular such that for $f = \hat g, \hat r_{{\bm \pi}}$ or $\Ec( \cdot,{\bm \pi})$, the PDE $\frac{\partial V}{\partial t}(t, \mu) + \Tc^{{\bm \pi}}V(t, \mu)= 0$, $t \in [0, t']$, with the terminal condition $V(t', \mu) = f(t', \mu)$ has a classical solution $V_f \in C^{1, 2}([0, t'] \times \Pc_2(\R^d))$ satisfying
    \begin{align*}
    & |\frac{\partial V_f}{\partial t}(t, \mu) - \frac{\partial V_f}{\partial t}(t, \mu')| + |\partial_\mu V_f(t, \mu)(x) - \partial_\mu V_f(t, \mu')(x')|\\
    & \;\;\; + |\partial_x\partial_\mu V_f(t, \mu)(x) - \partial_x\partial_\mu V_f(t, \mu')(x')| \leq C_\Pi\Big(|x - x'| + \Wc_2(\mu, \mu')\Big),\\
    &|\partial_\mu^2 V_f(t, \mu)(x, y) - \partial_\mu^2 V_f(t, \mu')(x', y')| \leq C_{\Pi}\Big(|x-x'| + |y-y'| + \Wc_2(\mu, \mu')\Big),
    \end{align*}
where the operator $\Tc^{\bm \pi}$ is defined by
    \begin{align*}
    \Tc^{\bm \pi}V(t, \mu) =& \int_{\R^d} \Big(b_{\bm \pi}(t, x, \mu)\trans \partial_\mu V(t, \mu)(x) + \frac{1}{2} {\rm Tr}\big(\sigma_{\bm \pi}\sigma_{\bm \pi}\trans + \sigma_{o, \bm \pi} \sigma_{o, {\bm \pi}}\trans\big)(t, x, \mu) \partial_x\partial_\mu V(t, \mu)(x)\Big)\mu(dx) \nonumber\\
&+ \frac{1}{2}\int_{\R^d \times \R^d} {\rm Tr}\Big(\sigma_{o, {\bm \pi}}(t, x, \mu)\sigma_{o, {\bm \pi}}(t, x', \mu) \trans \partial_\mu^2 V(t, \mu)(x, x')\Big)\mu(dx) \otimes \mu(dx').
    \end{align*}
\end{Assumption}

\begin{Proposition}\label{prop:approximation}Under Assumptions \ref{ass:b-sigma-sigmao} and \ref{ass:approximation}, we have $|J^{\Dc}(t, \mu; {\bm \pi}) - \tilde J(t, \mu; {\bm \pi})| \leq C |\Dc|^{1/2}$,
where the constant $C$ depends on $b, \sigma, \sigma_o, r$, and $\Pi$. Consequently, $J(t, \mu; {\bm \pi}) = \tilde J(t, \mu; {\bm \pi})$.
\end{Proposition}

The proof of Proposition \ref{prop:approximation} relies on the following lemma.
\begin{Lemma}\label{lem:approximation}
Let Assumptions \ref{ass:b-sigma-sigmao}, and \ref{ass:approximation} hold. Then for $f$ in Assumption \ref{ass:approximation}, there exists a constant $C$ (depending only on $T$, $t$, $\gamma$, $b$, $\sigma$, $\sigma_0$, $\Pi$ and $f$) such that for all grids $\Dc$,
\begin{align}
\sup_{s \in [t, T]} \big|\E^e[f(\mu_s^{\Dc, {\bm \pi}}) - f(\mu_s^{\bm \pi})]\big| \leq C |\Dc|^{1/2}.
\end{align}
\end{Lemma}

\begin{proof}[Proof of Lemma \ref{lem:approximation}]
Without loss of generality, let us assume that $t' = s_i$. As ${\bm \pi}$ is fixed throughout the proof, we do not write the superscript $\bm \pi$ in $X_s^{\Dc, \bm \pi}$, $\mu_s^{\Dc, \bm \pi}$ and $\mu_s^{\bm \pi}$.  In view that $V_f$ is the solution of $\frac{\partial V}{\partial t}(t, \mu)+ \Tc^{{\bm \pi}} V(t, \mu) = 0$, $t \in [0, s_i]$, $V(s_i, \mu) = f(\mu)$, we have $\E^e[f(\mu_{s_i})] = V_f(t, \mu)$ by the Feynman-Kac's formula. It follows that
\begin{align*}
&\E^e[f(\mu_{s_i}^{\Dc}) - f(\mu_{s_i})] = \E^e[V_f(s_i, \mu_{s_i}^{\Dc}) - V_f(t, \mu)]
= \sum_{j=0}^{i-1} \E^e[V_f(s_{j+1}, \mu_{s_{j+1}}^{\Dc}) - V_f(s_j, \mu_{s_j}^{\Dc})]= \sum_{j=0}^{i-1} e_j.
\end{align*}
It thus suffices to estimate the term $e_j$.  Applying It\^o's lemma to $V_f(s, \mu_s^{\Dc})$ between $s_j$ and $s_{j+1}$ and taking the expectation on both sides, we get that
\begin{align*}
e_j= & \E^e\Big[\int_{s_j}^{s_{j+1}} \Big(\frac{\partial  V_f}{\partial t}(s, \mu_s^{\Dc}) + b(s, X_s^{\Dc}, \mu_s^{\Dc}, a_{s_j})\trans\partial_\mu V_f(s, \mu_s^{\Dc})(X_s^{\Dc})\\
& + \frac{1}{2} {\rm Tr}\big(\sigma\sigma + \sigma_o\sigma_o\trans\big)(s, X_s^{\Dc}, \mu_s^{\Dc}, a_{s_j}) \partial_x\partial_\mu V_f(s, \mu_s^{\Dc})(X_s^{\Dc}))\Big)ds\Big]\\
& + \frac{1}{2}\E^e \bar \E^e \Big[\int_{s_j}^{s_{j+1}} {\rm Tr}\big(\sigma_0(s, X_s^{\Dc}, \mu_s^{\Dc}, a_{s_j}) \sigma_0\trans(s, \bar X_s^{\Dc}, \mu_s^{\Dc}, \bar a_{s_j})\partial_\mu^2 V_f(s, \mu_s^{\Dc})(X_s^{\Dc}, \bar X_s^{\Dc})\big) ds\Big].
\end{align*}
On the other hand, $\big(\frac{\partial V_f}{\partial t}(s_j, \mu_{s_j}^{\Dc}) + \Tc^{\bm \pi} V_f(s_j, \mu_{s_j}^{\Dc})\big) \cdot (s_{j+1} - s_j) = 0$, for $0 \leq j \leq i-1$. Subtracting this term on both sides of the above equation, we get that
\begin{align*}
e_j = & \E^e\Big[\int_{s_j}^{s_{j+1}} \Big(\frac{\partial V_f}{\partial t}(s, \mu_s^{\Dc}) -  \frac{\partial V_f}{\partial t}(s_j, \mu_{s_j}^\Dc) + b(s, X_s^{\Dc}, \mu_s^{\Dc}, a_{s_j})\trans\partial_\mu V_f(s, \mu_s^{\Dc})(X_s^{\Dc})\\
& -  b(s_j, X_{s_j}^{\Dc}, \mu_{s_j}^{\Dc}, a_{s_j})\trans\partial_\mu V_f(s_j, \mu_{s_j}^{\Dc})(X_{s_j}^{\Dc})  + \frac{1}{2}{\rm Tr}\big(\sigma\sigma + \sigma_o\sigma_o\trans\big)(s, X_s^{\Dc}, \mu_s^{\Dc}, a_{s_j}) \partial_x\partial_\mu V_f(s, \mu_s^{\Dc})(X_s^{\Dc})\\
& - \frac{1}{2}{\rm Tr}\big(\sigma\sigma + \sigma_o\sigma_o\trans\big)(s_j, X_{s_j}^{\Dc}, \mu_{s_j}^{\Dc}, a_{s_j}) \partial_x\partial_\mu V_f(s_j, \mu_{s_j}^{\Dc})(X_{s_j}^{\Dc})\Big)ds\Big]\\
& + \frac{1}{2}\E^e \bar \E^e \Big[\int_{s_j}^{s_{j+1}} \Big({\rm Tr}\big(\sigma_o(s, X_s^{\Dc}, \mu_{s}^\Dc, a_{s_j}) \sigma_o\trans(s, \bar X_s^{\Dc}, \mu_{s}^\Dc, \bar a_{s_j})\partial_\mu^2 V_f(s, \mu_{s}^\Dc)(X_s^{\Dc}, \bar X_s^{\Dc})\big)\\
& - {\rm Tr}\big(\sigma_o(s_j, X_{s_j}^{\Dc}, \mu_{s_j}^\Dc, a_{s_j})\sigma_o\trans(s_j, \bar X_{s_j}^{\Dc}, \mu_{s_j}^\Dc, \bar a_{s_j})\partial_\mu^2 V_f(s_j, \mu_{s_j}^\Dc)(X_{s_j}^{\Dc}, \bar X_{s_j}^{\Dc})\big)\Big) ds\Big].
\end{align*}
Thanks to the continuity of $X_s^{\Dc}$,
we have $\E^e\left[\Wc_2^2(\mu_s^{\Dc}, \mu_{s'}^{\Dc})\right] \leq  \E^e\left[|X_s^{\Dc} - X_{s'}^{\Dc}|^2\right] \leq C|s -s'|$.
Combining it with the Lipschitz continuity on $b$, $\sigma$, $\sigma_0$, $\frac{\partial V_f}{\partial t}$, $\partial_\mu V_f$, $\partial_x\partial_\mu V_f$ and $\partial_\mu^2 V_f$, we conclude that $|e_j| \leq C (s_{j+1} - s_j) |\Dc|^{1/2}$.
We then get the desired result that $\sum_{j=0}^{i-1} |e_j| \leq C(T-t) |\Dc|^{1/2}$.
\end{proof}

We next proceed the proof of Proposition \ref{prop:approximation}.
\begin{proof}[Proof of Proposition \ref{prop:approximation}] Note that
\begin{align*}
&J^{\Dc} (t, \xi; {\bm \pi}) - \tilde J(t, \mu; {\bm \pi}) \\
& = \E^e\left[e^{-\beta(T -t)}\big(\hat g(\mu_T^{\Dc}) - \hat g(\mu_T)\big)\right] + \E^e\left[\sum_{i=0}^{n-1} \int_{s_i}^{s_{i+1}} e^{-\beta (s -t )}\Big(r(s, X_s^{\Dc}, \mu_s^{\Dc}, a_{s_i}) - \hat r_{{\bm \pi}}(s, \mu_s)\Big)ds\right]\\
& \;+ \gamma \sum_{i=0}^{n-1} \E^e\left[\int_{s_i}^{s_{i+1}} e^{-\beta (s -t )}\Big(\mathcal{E}(s_i, \mu_{s_i}^{\Dc}, {\bm \pi}) - \mathcal{E}(s, \mu_s, {\bm \pi})\Big)ds\right]=: I + \sum_{i=0}^{n-1} II^i + \gamma \sum_{i=0}^{n-1} III^i.
\end{align*}
By Lemma \ref{lem:approximation} with $f =\hat g$,  we have that
\begin{align}\label{estimateI}
|I| \leq C |\Dc|^{1/2}.
\end{align}
For the term $\sum_{i=0}^{n-1} II_i$, it holds that
\begin{align}
II^i =& \E^e\left[\int_{s_i}^{s_{i+1}} e^{-\beta (s -t )}\Big(r(s, X_s^{\Dc}, \mu_s^{\Dc}, a_{s_i}) - r(s_i, X_{s_i}^{\Dc}, \mu_{s_i}^{\Dc}, a_{s_i})\Big)ds\right] \nonumber\\
& + \E^e\left[\int_{s_i}^{s_{i+1}} e^{-\beta (s -t )}\Big(\hat r_{{\bm \pi}}(s_i, \mu_{s_i}^{\Dc}) - \hat r_{{\bm \pi}}(s_i, \mu_s)\Big)ds\right]\\
& + \E^e \left[\int_{s_i}^{s_{i+1}}e^{-\beta (s -t )} \Big(\hat r_{{\bm \pi}}(s_i, \mu_{s_i})- \hat r_{{\bm \pi}}(s, \mu_s)\Big)ds\right] = : II_1^i + II_2^i + II_3^i,\nonumber
\end{align}
where in the term $II_2^i$, we have $\E^e[r(s_i, X_{s_i}^{\Dc}, \mu_{s_i}^{\Dc}, a_{s_i})] = \E^e[\hat r_{{\bm \pi}}(s_i, \mu_{s_i}^{\Dc})]$ by the tower property. For the term $II_1^i$, we obtain by Assumption \ref{ass:b-sigma-sigmao} (iii) that
\begin{align}\label{esimateII1i}
|II_1^i| \leq & \E^e\biggl[\int_{s_i}^{s_{i+1}} e^{-\beta (s -t)} \Big(|X_s^{\Dc}| + |X_{s_i}^{\Dc}| + \|\mu_s^{\Dc}\|_2 + \|\mu_{s_i}^{\Dc}\|_2\Big) \Big(|s- s_i|\\
& + |X_s^{\Dc}- X_{s_i}^{\Dc}| + \Wc_2(\mu_s^{\Dc},  \mu_{s_i}^{\Dc})\Big)ds\biggl] \leq C(s_{i+1} - s_i) \cdot |\Dc|^{1/2}. \nonumber
\end{align}
Similarly, by Lipschitz continuity of $\mu_s$ in time and local Lipschitz continuity on $\hat r_{\bm \pi}$, it holds that
\begin{align}\label{estimateII3i}
|II_3^i| \leq C (s_{i+1} - s_i) \cdot |\Dc|^{1/2}.
\end{align}
It follows from Lemma \ref{lem:approximation} with $f = \hat r_{{\bm \pi}}$ that
\begin{align}\label{estimateII2i}
|II_2^i| \leq C (s_{i+1} - s_i) \cdot |\Dc|^{1/2}.
\end{align}
Finally, for the term $III^i$, we can deduce that
\begin{align}\label{estimateIIIi}
|III^i| \leq& \E^e\left[\int_{s_i}^{s_{i+1}} e^{-\beta (s -t )}\Big(\mathcal{E}(s_i, \mu_{s_i}^{\Dc}, {\bm \pi}) - \mathcal{E}(s_i, \mu_{s_i},{\bm \pi}) + \mathcal{E}(s_i, \mu_{s_i},{\bm \pi}) - \mathcal{E}(s, \mu_s, {\bm \pi})\Big)ds\right] \nonumber\\
& \leq C (s_{i+1} - s_i) \cdot |\Dc|^{1/2},
\end{align}
where in the last inequality, we have used Lemma \ref{lem:approximation}, the local Lipschitz continuity of $\Ec(\cdot, {\bm \pi})$ and the continuity of $\mu_s$ in time $s$. Combining \eqref{estimateI}, \eqref{esimateII1i}, \eqref{estimateII3i}, \eqref{estimateII2i}, and \eqref{estimateIIIi}, we conclude the result.
\end{proof}

The equivalence between value functions in Proposition \ref{prop:approximation} allows us to use the relaxed control formulation for the derivation of HJB equation and some theoretical analysis afterwards.

\subsection{First-order condition and policy improvement}


We first give a PDE characterization of $\tilde J(\cdot, \cdot; {\bm \pi})$ in \eqref{equ:exploratory_average_value_function} based on the Feynman-Kac formula. It is assumed that $\tilde J(\cdot, \cdot; {\bm \pi}) \in \Cc^{1, 2}([0, T] \times \Pc_2(\R^d))$ in Lemma \ref{lemma:regularityJ} to avoid heavy technicalities. Interested readers can generalize Malliavin calculus arguments in \cite{BCL2021, CCD2022,crisanMcMurray2018} to the common noise setting to investigate sufficient conditions for the regularity of $\tilde J(\cdot, \cdot; {\bm \pi})$.

\begin{Lemma}\label{lemma:regularityJ} Assume that the function $\tilde J(\cdot, \cdot; {\bm \pi})$ belongs to $\Cc^{1, 2}([0, T] \times \Pc_2(\R^d))$. Then it satisfies the following PDE
\begin{align}
&\frac{\partial \tilde J}{\partial t}(t, \mu; {\bm \pi})  +
\int_{\R^d \times \Ac}H\big(t, x, \mu, a,\partial_\mu \tilde J(t, \mu; {\bm \pi})(x), \partial_x\partial_\mu \tilde J(t, \mu; {\bm \pi})(x)\big){\bm \pi}(a|t, x, \mu)da\mu(dx) \nonumber\\
&+ \frac{1}{2}\int_{\R^d \times \R^d} {\rm Tr}\Big(\sigma_{o, {\bm \pi}}(t, x, \mu)\sigma_{o, {\bm \pi}}(t, x', \mu) \trans \partial_\mu^2 \tilde J(t, \mu; {\bm \pi})(x, x')\Big)\mu(dx) \otimes \mu(dx') \label{equ:dynamic-programming-equation}\\
& - \beta \tilde J(t, \mu; {\bm \pi}) + \gamma \mathcal{E}(t, \mu, {\bm \pi})=0, \nonumber
\end{align}
where the Hamiltonian operator $H$ is defined by
\begin{align}\label{operatorH}
H(t, x, \mu, a, p, q) & :=  b(t, x, \mu, a) \trans p + \frac{1}{2}{\rm Tr}\Big(\big(\sigma\sigma\trans + \sigma_o\sigma_o\trans\big)(t, x, \mu, a)
q\Big) + r(t, x, \mu, a).
\end{align}
\end{Lemma}

\begin{proof}
From the flow property of $\mu_s^{t-h, \mu} = \mu_s^{t, \mu_t^{t-h, \mu}}$ for any $0 \leq h \leq t$
\begin{align*}
\tilde J(t-h, \mu; {\bm \pi}) &= e^{-\beta h} \E^e\Big[\int_{t - h}^t e^{-\beta(s - t)} \big(r_{\bm \pi}(s, \mu_s)+ \gamma \Ec(s, \mu_s, {\bm \pi})\big)ds\Big]+ e^{-\beta h} \E^e[\tilde J(t, \mu_t; {\bm \pi})],
\end{align*}
In view that $\tilde J(\cdot, \cdot; {\bm \pi}) \in \Cc^{1, 2}([0, T] \times \Pc_2(\R^d))$, applying It\^o's formula in \cite{CarD2} leads to
\begin{align}\label{equ:regularity-J-t}
&h^{-1} \big(\tilde J(t-h, \mu; {\bm \pi}) - \tilde J(t, \mu; {\bm \pi})\big) \\
 =& h^{-1} e^{-\beta h} \E^e\Big[\int_{t - h}^t e^{-\beta(s - t)} \big(\hat r_{\bm \pi}(s, \mu_s)+ \gamma \Ec(s,\mu_s, {\bm \pi})\big)ds\Big] \nonumber\\
& + h^{-1} e^{-\beta h}\E^e\big[\tilde J(t, \mu_t; {\bm \pi}) - \tilde J(t, \mu; {\bm \pi})\big] + h^{-1} (e^{-\beta h} - 1) \tilde J(t, \mu; {\bm \pi}) \nonumber\\
= & h^{-1} e^{-\beta h} \E^e\Big[\int_{t - h}^t \mathscr{H}^\gamma(s, \mu_s, {\bm \pi}; {\bm \pi})ds\Big] + h^{-1} (e^{-\beta h} - 1) \tilde J(t, \mu; {\bm \pi}).\nonumber
\end{align}
Letting $h \to 0$, we obtain \eqref{equ:dynamic-programming-equation}.
\end{proof}

Under proper conditions, the optimal value function $\tilde J^*$ in \eqref{average:optimal-value} satisfies the HJB equation
\begin{align}
& \frac{\partial \tilde J^*}{\partial t}(t, \mu) + \sup_{{\bm \pi} \in \Pc(\Ac|\R^d)} \biggl\{\int_{\R^d \times \Ac} H\big(t, x, \mu, a,\partial_\mu \tilde J^*(t, \mu)(x), \partial_x\partial_\mu \tilde J^*(t, \mu)(x)\big){\bm \pi}(a|x)da\mu(dx) \nonumber\\
&+ \frac{1}{2}\int_{\R^d \times \R^d} {\rm Tr}\Big(\sigma_{o, {\bm \pi}}(t, x, \mu)\sigma_{o, {\bm \pi}}(t, x', \mu) \trans \partial_\mu^2 \tilde J^*(t, \mu)(x, x')\Big)\mu(dx) \otimes \mu(dx') \label{HJB}\\
&+ \gamma \mathcal{E}(t, \mu, {\bm \pi})\biggl\} - \beta \tilde J^*(t, \mu) =0. \nonumber
\end{align}
Here, for fixed $(t, \mu) \in [0, T] \times \Pc_2(\R^d)$, we will frequently identify a policy ${\bm \pi} \in \Pi$ with the probability transition  kernel ${\bm \pi}\in \Pc_{ac}(\Ac|\R^d)$ by abuse of notations, and hereafter we will not distinguish $\Pi$ and $\Pc_{ac}(\Ac|\R^d)$ when there is no confusion.

Note that the sup operator in \eqref{HJB} for the optimal value function leads to a fully nonlinear PDE in the Wasserstein space. The conventional policy iteration is to linearize \eqref{HJB} from a given policy and hope to iterate the policy to reach the optimal one. This motivates us to consider the functional $\mathscr{H}: [0, T] \times \Pc_2(\R^d) \times \Pc_{ac}(\Ac|\R^d) \to \R$ as the {\it integrated Hamiltonian} under the fixed policy ${\bm \pi}$ defined by
\begin{align}\label{equ:F}
\mathscr{H}(t, \mu, {\bm h}; {\bm \pi}) &:= \int_{\R^d \times \Ac} H\big(t, x, \mu, a,\partial_\mu \tilde J(t, \mu; {\bm \pi})(x), \partial_x\partial_\mu \tilde J(t, \mu; {\bm \pi})(x)\big) {\bm h}(a| x) da \mu(dx)\\
&\;\;\; + \frac{1}{2}\int_{\R^d \times \R^d}{\rm Tr}\Big(\sigma_{o, {\bm h}}(t, x, \mu)\sigma_{o, {\bm h}}(t, x', \mu) \trans \partial_\mu^2 \tilde J(t, \mu; {\bm \pi})(x, x')\Big)\mu(dx) \otimes \mu(dx'). \nonumber
\end{align}
Due to the dependence of the coefficient $\sigma_o$ on the action $a$, $\mathscr{H}(t, \mu, {\bm h}; {\bm \pi})$ is nonlinear in ${\bm h}$. For the fixed policy ${\bm \pi}$, we consider the following entropy-regularized optimization problem over the space of policies ${\bm h} $ (as the probability transition kernel ${\bm h}\in \Pc_{ac}(\Ac|\R^d)$ with the fixed $(t,\mu)$):
\begin{align}\label{equ:regularized_minimization}
\sup_{{\bm h} \in \Pc_{ac}(\Ac|\R^d)} \Big(\mathscr{H}(t, \mu, {\bm h}; {\bm \pi}) + \gamma \mathcal{E}(t, \mu, {\bm h})\Big) =: \sup_{{\bm h} \in \Pc_{ac}(\Ac|\R^d)} \mathscr{H}^\gamma(t, \mu, {\bm h}; {\bm \pi}).
\end{align}

Given the current policy $\bm \pi$, to characterize the maximizer in \eqref{equ:regularized_minimization}, called the optimal one-step iterated policy, we adopt two notions of concavity with respect to measures, namely, the classical concavity and the displacement concavity  in \cite{McCann97} and \cite{Villani09} to our setting with respect to probability transition kernels.

\begin{Assumption}\label{assum:2nd-mu-derivative-J}
For a fixed ${\bm \pi} \in \Pi$ and $(t, \mu) \in [0, T] \times \Pc_2(\R^d)$, the integrated Hamiltonian $\mathscr{H}$ in \eqref{equ:F} satisfies either of the following two conditions: for ${\bm h}_0, {\bm h}_1$ in $\Pc_{ac}(\Ac|\R^d)$ and any $\theta \in [0, 1]$
\begin{itemize}
 \item [(i)] $\mathscr{H}$ is concave in the classical sense
\begin{align*}
\mathscr{H}(t, \mu, \theta {\bm h}_0 + (1 - \theta){\bm h}_1; {\bm \pi}) \geq \theta \mathscr{H}(t, \mu, {\bm h}_0; {\bm \pi}) + (1 -\theta)\mathscr{H}(t, \mu, {\bm h}_1; {\bm \pi}),
\end{align*}
 \item [(ii)] $\mathscr{H}$ is concave in ${\bm h} \in \Pc_{ac}(\Ac|\R^d)$ in the displacement concave sense if there exists a geodesic ${\bm h}_\theta: = \P^2_{\theta \phi_{{\bm h_0}}(\cdot, U) + (1 - \theta) \phi_{{\bm h}_1}(\cdot, U)}$
 \begin{align*}
 \mathscr{H}(t, \mu, {\bm h}_\theta; {\bm \pi}) \geq \theta \mathscr{H}(t, \mu, {\bm h}_0; {\bm \pi}) + (1 - \theta) \mathscr{H}(t, \mu, {\bm h}_1; {\bm \pi}),
 \end{align*}
 where for any ${\bm h} \in \Pc_{ac}(\Ac|\R^d)$, $\phi_{{\bm h}}(\cdot, U)$ is a random variable on $(\Omega^2, \Fc^2, \P^2)$ with $\P^2_{\phi_{{\bm h}}(x, U)} = {\bm h}(\cdot|x)$.
\end{itemize}
\end{Assumption}

\begin{Remark} The difference between the classical concavity and the displacement concavity lies in the means of interpolation.
Both notions have their own merits:
\begin{itemize}
\item When $\sigma_o$ is not controlled, $\mathscr{H}$  is linear in ${\bm h} \in \Pc_{ac}(\Ac|\R^d)$ and hence concave in the classical sense. However, to make sure that $\mathscr{H}$ is displacement concave in ${\bm h} \in \Pc_{ac}(\Ac|\R^d)$, we have to assume that $H(t, x, \mu, a, p, q)$ is concave in $a \in \Ac$. Therefore, it would be better to utilize the classical concavity rather than the displacement concavity in this case.

\item When $\sigma_o$ is controlled,  one necessary and sufficient condition that $\mathscr{H}$ in \eqref{equ:F} is concave in ${\bm h}$ is that $\partial_\mu^2 J(t, \mu; {\bm \pi})$ satisfies for any ${\bm h}, {\bm h}' \in \Pc_{ac}(\Ac|\R^d)$
\begin{align*}
 \int_{\R^d \times \R^d}{\rm Tr}\Big(\big(\sigma_{o,  {\bm h}} -\sigma_{o, {\bm h}'}\big)(t, x, \mu)\big(\sigma_{o, {\bm h}} - \sigma_{o, {\bm h}'}\big)(t, x', \mu) \trans \partial_\mu^2 J(t, \mu; {\bm \pi})(x, x')\Big)\mu(dx) \otimes \mu(dx') \leq 0.
\end{align*}
However, some classical LQ-MFC problems, such as the mean-variance optimization problems, do not satisfy the above inequality. Instead, one can easily check that the displacement concavity condition holds for LQ-MFC problems; see section \ref{sec:LQ} for more details.
\end{itemize}

\end{Remark}


Essentially, we can use the derivative of $\mathscr{H}^\gamma$ with respect to the policy to characterize the improved policy. To this end, in the same spirit of the linear functional derivative with respect to the probability measure, see Definition 5.43 \cite{CarD}, we consider the following definition of the partial linear functional derivative  with respect to the probability transition kernel, see section 2.1 in \cite{CKR2023}.

\begin{Definition}[Partial linear functional derivative ] \label{def:linear-derivative}
Fix $\mu \in \Pc_2(\R^d)$. The functional $\frac{\delta G}{\delta {\bm h}}: \Pc_{ac}(\Ac|\R^d) \times \R^d \times \Ac \to \R$ is said to be a partial linear functional derivative  of $G: \Pc_{ac}(\Ac|\R^d) \to \R$ with respect to the probability transition kernel ${\bm h} \in \Pc_{ac}(\Ac|\R^d)$ if for any ${\bm h}, {\bm h}' \in \Pc_{ac}(\Ac|\R^d)$,
\begin{align*}
G({\bm h}') - G({\bm h}) =  \int_0^1 \int_{\R^d}\int_ {\Ac} \frac{\delta G}{\delta {\bm h}}((1 -\lambda){\bm h} + \lambda{\bm h}')(x, a)({\bm h}' - {\bm h})(a|x)da \mu(dx) d\lambda.
\end{align*}
Moreover, there exists a constant $C >0$, possibly depending on $\mu$, such that $$\sup_{{\bm h} \in \Pc_{ac}(\Ac|\R^d)} \big|\frac{\delta G}{\delta {\bm h}}({\bm h})(x, a)\big| \leq C(1 + |x|^2 + |a|^2).$$
\end{Definition}

\begin{Remark}
The partial linear functional derivative  with respect to the  probability transition kernel is unique up to an additive function $\kappa(x, \mu)$ satisfying $\int_{\R^d} \kappa(x, \mu)\mu(dx) = C$ for some constant. See  \cite{BCL2021} for the partial $L$-derivative with respect to probability transition kernels.
\end{Remark}

Given the current policy ${\bm \pi} \in \Pi$, the next result gives the existence, uniqueness and the first-order condition of the optimal one-step iterated policy for the problem  \eqref{equ:regularized_minimization}.

\begin{Theorem}\label{lemma:first-order}
Let Assumptions \ref{ass:b-sigma-sigmao} and \ref{assum:2nd-mu-derivative-J} hold. Assume that $J(\cdot, \cdot; {\bm \pi})\in \Cc^{1, 2}([0, T] \times \Pc_2(\R^d))$.  Given ${\bm \pi} \in \Pi$ and $(t, \mu) \in [0, T] \times \Pc_{2 + \delta}(\R^d)$ for some $\delta >0$.  The entropy-regularized integrated Hamiltonian $\mathscr{H}^\gamma(t, \mu, {\bm h}; {\bm \pi})$ in \eqref{equ:regularized_minimization} has a unique maximizer ${\bm h}^* \in \Pi$ if and only if ${\bm h}^*$ satisfies
\begin{align}\label{equ:first-order}
\frac{\delta \mathscr{H}}{\delta {\bm h}}(t, \mu,  {\bm h}^*; {\bm \pi}) (x, a)- \gamma \log {\bm h}^*(a|t, x, \mu)   = \kappa(t, x, \mu),
\end{align}
where $\frac{\delta \mathscr{H}}{\delta {\bm h}}$ is given by
\begin{align}\label{equ:derivativeF}
\frac{\delta \mathscr{H}}{\delta {\bm h}}(t, \mu, {\bm h}; {\bm \pi})(x, a) = & H\big(t, x, \mu, a,\partial_\mu J(t, \mu; {\bm \pi})(x), \partial_x\partial_\mu J(t, \mu; {\bm \pi})(x)\big)\\
&+\int_{\R^d} {\rm Tr}\Big(\sigma_o(t, x, \mu, a) \sigma_{o, {\bm h}}(t, x', \mu)\trans\partial_\mu^2 J(t, \mu; {\bm \pi})(x, x')\Big)\mu(dx'). \nonumber
\end{align}
 Or equivalently, ${\bm h}^*$ is the fixed point of
of $\Phi_{\bm \pi}: \Pi \to \Pi$ defined by
\begin{align}\label{equ:mapI}
 \Phi_{\bm \pi}(\bm h)(a|t, x, \mu)  = \frac{\exp\Big\{\frac{1}{\gamma} \frac{\delta \mathscr{H}}{\delta {\bm h}}(t, \mu, \bm h; {\bm \pi})(x, a)\Big\}}{\int_{\Ac} \exp\Big\{\frac{1}{\gamma} \frac{\delta \mathscr{H}}{\delta {\bm h}}(t, \mu, {\bm h}; {\bm \pi})(x, a)\Big\}da}.
\end{align}
Consequently, ${\bm h}^*$ is a map of ${\bm \pi}$ and we denote by ${\bm h}^* = \mathcal{I}({\bm \pi})$.
\end{Theorem}

\begin{proof}
\noindent {\it Step-1}.\;  We first show the existence and uniqueness of the maximizer ${\bm h}^*$ of \eqref{equ:regularized_minimization}. In this step, we change the admissible policy space from $\Pc_{ac}(\Ac|\R^d)$ to $\Pc(\Ac|\R^d)$, which is more convenient for the compactness and does not affect the result. This is because for those policies not in $\Pc_{ac}(\Ac|\R^d)$, we set $\Ec(t, \mu, {\bm h}) = -\infty$ by convention. Thereby, if a maximizer ${\bm h}$ exists, it belongs to $\Pc_{ac}(\Ac|\R^d)$.

Define $\Vc_\mu: = \{{\bm \nu} \in \Pc_2(\R^d \times \Ac): {\bm \nu}(\cdot,\R^d) = \mu\}$ as the space of probability measures on $\R^d \times \Ac$ whose first marginal $\mu$ is fixed.
Then, the one-to-one correspondence between $\Pc_{ac}(\Ac|\R^d)$ and $\Vc_\mu$ holds, i.e.,  for each ${\bm \nu} \in \Vc_\mu$, there exists a ${\bm h}(da|x) \in \Pc(\Ac|\R^d)$ by disintegration such that ${\bm \nu}(dx, da) = {\bm h}(da|x) \mu(dx)$. Conversely, each ${\bm h}(da|x) \in \Pc(\Ac|\R^d)$ induces a probability measure ${\bm \nu} \in \Vc_\mu$. Thanks to the equivalence between $\Vc_\mu$ and $\Pc(\Ac|\R^d)$, we may rewrite $\mathscr{H}^\gamma(t, \mu, {\bm h}; {\bm \pi})$ as a functional of ${\bm \nu} \in \Vc_\mu$ with a slight abuse of notation and
\begin{align}
\Ec(t, \mu, {\bm h}) = \Ec(t, \mu, {\bm \nu}), \; \mathscr{H}(t, \mu, {\bm \nu}; {\bm \pi}) = \mathscr{H}(t, \mu, {\bm h}; {\bm \pi}),
\end{align}
with
\begin{align*}
&\mathscr{H}(t, \mu, {\bm \nu}; {\bm \pi})
=\int_{\R^d \times \Ac} H\big(t, x, \mu, a,\partial_\mu \tilde J(t, \mu; {\bm \pi})(x), \partial_x\partial_\mu \tilde J(t, \mu; {\bm \pi})(x)\big) {\bm \nu}(dx, da)\\
& + \frac{1}{2}\int_{\R^{2d} \times \Ac^2}{\rm Tr}\Big(\sigma_{o}(t, x, \mu, a)\sigma_{o}(t, x', \mu, a') \trans \partial_\mu^2 \tilde J(t, \mu; {\bm \pi})(x, x')\Big){\bm \nu}(dx, da) \otimes {\bm \nu}(dx', da'),
\end{align*}
and whenever $\nu$ is absolutely continuous with respect to $\mu(dx) da$
\begin{align*}
\Ec(t, \mu, {\bm \nu}) & = -\int_{\R^d \times \Ac} \log \frac{{\bm \nu}(dx, da)}{\mu(dx)da} {\bm \nu}(dx, da).
\end{align*}
Otherwise, we set $\Ec({\bm \nu}) = - \infty$. Hence the optimization problem \eqref{equ:regularized_minimization} becomes
 \begin{align} \label{equ:reformulated-regularized_minimization}
 \sup_{{\bm \nu} \in \Vc_\mu} \mathscr{H}^\gamma(t, \mu, {\bm \nu}; {\bm \pi}).
 \end{align}
The next step is to reduce the problem \eqref{equ:reformulated-regularized_minimization} to maximizing an upper semicontinuous function on a {\it compact} set $\Sc_\mu$ of $\Vc_\mu$ under the $\Wc_2$ metric. First, note that there exists some $\bar{\bm \nu} \in \Vc_\mu$ such that $\mathscr{H}^\gamma(t, \mu, \bar {\bm \nu}; {\bm \pi}) < +\infty$ and let us introduce a subset $\Sc_\mu$ of $\Vc_\mu$
\begin{align*}
\Sc_\mu: = \Big\{{\bm \nu} \in \Vc_\mu: \gamma \mathcal{E}({\bm \nu}) \geq \mathscr{H}^\gamma(t, \mu, \bar {\bm \nu}; {\bm \pi}) - \sup_{{\bm \nu}' \in \Vc_\mu} \mathscr{H}(t, \mu, {\bm \nu}'; {\bm \pi})\Big\}.
\end{align*}
It follows from  the definition of $\Sc_\mu$ that $\mathscr{H}^\gamma(t, \mu, {\bm \nu}; {\bm \pi}) \leq \mathscr{H}^\gamma(t, \mu, \bar {\bm \nu}; {\bm \pi})$ for any ${\bm \nu} \notin \Sc_\mu$, hence $\sup_{{\bm \nu} \in \Vc_\mu} \mathscr{H}^\gamma(t, \mu, {\bm \nu}; {\bm \pi}) = \sup_{{\bm \nu} \in \Sc_\mu} \mathscr{H}^\gamma(t, \mu, {\bm \nu}; {\bm \pi})$. As the sublevel set of $-\mathcal{E}({\bm \nu})$ is weakly compact by Lemma 1.4.3 in \cite{Dupuis2011}, $\Sc$ is weakly compact.
Furthermore, by Definition \ref{def:pi} (iii), we have
\begin{align*}
\sup_{{\bm \nu} \in \Sc_\mu} \int_{\R^d \times \Ac} (|x|^2 + |a|^2)^{(2 + \delta)/2} {\bm \nu}(dx, da)\leq  C \int_{\R^d}\Big(|x|^{2 + \delta} + \sup_{{\bm h} \in \Pi} \int_{\Ac}|a|^{2 + \delta} {\bm h}(a|t, x, \mu)da \Big)\mu(dx)
 < +\infty.
\end{align*}
Theorem 5.5 in \cite{CarD} guarantees that $\Sc_\mu$ is compact in $\Pc_2(\R^d \times \Ac)$ under the $\Wc_2$ metric.

By Assumption \ref{ass:b-sigma-sigmao}, $H\big(t, x, \mu, a,\partial_\mu \tilde J(t, \mu; {\bm \pi})(x), \partial_x\partial_\mu \tilde J(t, \mu; {\bm \pi})(x)\big)$ is continuous in $(x, a)$ and square integrable.
Similarly, ${\rm Tr}\Big(\sigma_{o}(t, x, \mu, a)\sigma_{o}(t, x', \mu, a') \trans \partial_\mu^2 \tilde J(t, \mu; {\bm \pi})(x, x')\Big)$ is continuous in $(x, x', a, a')$ and square integrable.
Therefore, by Lemma A.3 in \cite{lacker2015}, $\mathscr{H}(t, \mu, {\bm \nu}; {\bm \pi})$ is continuous under the $\Wc_2$ metric. On the other hand, by Lemma 1.4.3 in \cite{Dupuis2011}, $\mathcal{E}(t, \mu, {\bm \nu})$ is upper semicontinuous on $\Pc_2(\R^d \times \Ac)$ under the weak topology and hence continuous under the $\Wc_2$ metric on $\Pc_2(\R^d \times \Ac)$. Therefore, $\mathscr{H}^\gamma$ is continuous and thus its supremum is attained in $\Sc_\mu$, which implies the existence of the maximizer of $\mathscr{H}^\gamma$. 

In view that $\mathscr{H}^\gamma$ is strictly concave or strictly displacement concave, the maximizer ${\bm \nu}^* \in \Vc_\mu$ is unique, which implies the uniqueness of the maximizer ${\bm h}^* \in \Pc(\Ac|\R^d)$ by disintegration. Note that $\Ec(t, \mu, {\bm \nu}^*) > -\infty$, it holds that ${\bm h}^* \in \Pc_{ac}(\Ac|\R^d)$.

\medskip

\noindent {\it Step-2}.\; We verify the sufficient and necessary condition of the first-order condition under the condition (i) when $\mathscr{H}^\gamma$ is strictly concave in ${\bm h} \in \Pc_{ac}(\Ac|\R^d)$.

\noindent {\it Step-2.1}\; Let us first prove the sufficient condition. Let ${\bm h}^*$ satisfy \eqref{equ:first-order}. Denote ${\bm h}^{\theta}: = (1 -\theta) {\bm h}^* + \theta {\bm h}$ for any ${\bm h} \in \Pc_{ac}(\Ac|\R^d)$ and $0 < \theta \leq 1$.
As $\mathscr{H}$ is concave in ${\bm h} \in \Pc_{ac}(\Ac|\R^d)$, we have
\begin{align}
\frac{1}{\theta}\big(\mathscr{H}(t, \mu, {\bm h}^\theta; {\bm \pi}) - \mathscr{H}(t, \mu, {\bm h}^*; {\bm \pi})\big)\leq &  \frac{1}{\theta} \int_{\R^d \times \Ac} \frac{\delta \mathscr{H}}{\delta {\bm h}}(t, \mu, {\bm h}^*)(x, a) ({\bm h}^\theta - {\bm h}^*)(a|x)da \mu(dx) \nonumber\\
= &\int_{\R^d \times \Ac}\frac{\delta \mathscr{H}}{\delta {\bm h}}(t, \mu, {\bm h}^*)(x, a) ({\bm h}- {\bm h}^*)(a|x)da \mu(dx). \label{proof:F}
\end{align}
Similarly, by the concavity of $\mathcal{E}(t, \mu, {\bm h})$ in ${\bm h}$, we get that
\begin{align}
\frac{\gamma}{\theta}\big(\mathcal{E}(t, \mu, {\bm h}^\theta) - \mathcal{E}(t, \mu, {\bm h}^*))\big) \leq & \frac{\gamma}{\theta} \int_{\R^d \times \Ac} \frac{\delta \mathcal{E}}{\delta {\bm h}}(t, \mu, {\bm h}^*)(x, a) ({\bm h}^\theta - {\bm h}^*)(a|x)da \mu(dx) \nonumber\\
=& -\gamma \int_{\R^d \times \Ac} \log{\bm h}^*(a|x) ({\bm h}- {\bm h}^*)(a|x)da \mu(dx).
\label{proof:mathcalE}
\end{align}
Summing \eqref{proof:F} and \eqref{proof:mathcalE}, we obtain that
\begin{align*}
& \frac{\mathscr{H}^\gamma(t, \mu, {\bm h}^\theta; {\bm \pi})- \mathscr{H}^\gamma(t, \mu, {\bm h}^*; {\bm \pi})}{\theta}\\
\leq & \int_{\R^d \times \Ac}\Big(\frac{\delta \mathscr{H}}{\delta {\bm h}}(t, \mu, {\bm h}^*; {\bm \pi})(x, a) - \gamma \log {\bm h}^*(a|x)\Big)({\bm h}- {\bm h}^*)(a|x)da \mu(dx)= 0.
\end{align*}
This implies that $\mathscr{H}^\gamma(t, \mu, {\bm h}^\theta; {\bm \pi})  \leq \mathscr{H}^\gamma(t, \mu, {\bm h}^*; {\bm \pi})$ for any ${\bm h} \in \Pi$ if $\theta = 1$.

\noindent {\it Step-2.2}.\; We then show the necessary condition. Let ${\bm h}^*$ be the maximizer of $\mathscr{H}^\gamma(t, \mu, {\bm h}; {\bm \pi})$. It then holds that
\begin{align*}
& \frac{\mathscr{H}^\gamma(t, \mu, {\bm h}^\theta; {\bm \pi}) - \mathscr{H}^\gamma(t, \mu, {\bm h}^*; {\bm \pi})}{\theta}\\
=& \frac{1}{\theta}\Big(\mathscr{H}(t, \mu, {\bm h}^\theta) - \mathscr{H}(t, \mu, {\bm h}^*)+ \gamma\big(\mathcal{E}(t, \mu, {\bm h}^\theta) - \mathcal{E}(t, \mu, {\bm h}^*)\big)\Big) \\
= & \frac{1}{\theta}\int_0^1 \int_{\R^d \times \Ac}\Big(\frac{\delta \mathscr{H}}{\delta {\bm h}} (t, \mu, {\bm h}^{\lambda, \theta}; {\bm \pi})(x, a) - \gamma \log {\bm h}^{\lambda, \theta} (a|x)\Big)({\bm h}^\theta - {\bm h})(a|x)da\mu(dx)d\lambda\\
= & \int_0^1  \int_{\R^d \times \Ac}\Big(\frac{\delta \mathscr{H}}{\delta {\bm h}} (t, x, \mu, {\bm h}^{\lambda, \theta}, a; {\bm \pi}) - \gamma \log{\bm h}^{\lambda, \theta}(a|x)\Big)({\bm h} - {\bm h}^*)(a|x)da\mu(dx),
\end{align*}
where we denote ${\bm h}^{\lambda, \theta} = \lambda {\bm h}^\theta + (1 - \lambda) {\bm h}^*$.

As $\frac{\delta \mathscr{H}}{\delta {\bm h}}$ in \eqref{equ:derivativeF} is continuous and integrable by Assumption \ref{ass:b-sigma-sigmao}, by dominated convergence theorem, we obtain that
\begin{align*}
& \lim_{\theta \to 0}\frac{\mathscr{H}^\gamma(t, \mu, {\bm h}^\theta; {\bm \pi})  - \mathscr{H}^\gamma(t, \mu, {\bm h}^*; {\bm \pi}) }{\theta}\\
= & \int_{\R^d \times \Ac}\Big(\frac{\delta \mathscr{H}}{\delta {\bm h}} (t, \mu, {\bm h}^*; {\bm \pi})(x, a) - \gamma \log{\bm h}^*(a|x)\Big)\big(\frac{{\bm h}}{{\bm h}^*} - 1\big){\bm h}^*(a|x)da\mu(dx) \leq 0,
\end{align*}
for any ${\bm h}$. By the arbitrariness of ${\bm h}$ and Definition \ref{def:pi} (i), it holds that $\frac{\delta \mathscr{H}}{\delta {\bm h}} (t, \mu, {\bm h}^*; {\bm \pi})(x, a) - \gamma \log{\bm h}^*(a|x)=\kappa(t, x, \mu)$.

\medskip

\noindent{\it Step-3}.\;When $\mathscr{H}^\gamma$ is strictly displacement concave, we verify the necessary and sufficient condition of the first-order condition. By the definition of displacement concavity,
$$[0, 1] \ni \theta \mapsto \mathscr{H}^\gamma(t, \mu, \P^2_{(1 -\theta) \phi_{{\bm h}^*}(\cdot, U) + \theta \phi_{{\bm h}}(\cdot, U)}; {\bm \pi})$$ is concave in $\theta \in [0, 1]$. Therefore, ${\bm h}^*$ is the maximizer if and only if for any $\phi_{\bm h}: \R^d \to \Ac$,
\begin{align*}
\frac{d}{d\theta}\mathscr{H}^\gamma(t, \mu, \P^2_{(1 -\theta) \phi_{{\bm h}^*}(\cdot, U) + \theta \phi_{{\bm h}}(\cdot, U)}; {\bm \pi})\big|_{\theta = 0}=0.
\end{align*}
Take $\phi_{\bm h} = \Tc \circ \phi_{{\bm h}^*}$ with the map $\Tc: \Ac \to \Ac$ being injective and surjective, and denote $\phi_{{\bm h}_\theta} := (1 - \theta) \phi_{{\bm h}^*} + \theta \Tc \circ \phi_{{\bm h}^*}$ and  ${\bm h}_\theta := \P^2_{\phi_{{\bm h}_\theta}(\cdot,U)}$. It then holds that
\begin{align*}
& \frac{d}{d\theta}\mathscr{H}^\gamma(t, \mu, {\bm h}_\theta; {\bm \pi})\big|_{\theta = 0} \\
=& \lim_{\theta \to 0} \frac{1}{\theta} \Big(\mathscr{H}^\gamma(t, \mu, {\bm h}_\theta; {\bm \pi}) - \mathscr{H}^\gamma(t, \mu, {\bm h}^*; {\bm \pi})\Big)\\
=& \lim_{\theta \to 0} \frac{1}{\theta}\int_0^1\int_{\R^d \times \Ac} \frac{\delta \mathscr{H}^\gamma}{\delta {\bm h}}(t, \mu, {\bm h}_{\theta, \lambda})(x, a) ({\bm h}_\theta - {\bm h}^*)(a|x)\mu(dx) d \lambda\\
=& \lim_{\theta \to 0} \frac{1}{\theta}\int_0^1\Big(\int_{\R^d \times \Ac} \frac{\delta \mathscr{H}^\gamma}{\delta {\bm h}}(t, \mu, {\bm h}_{\theta, \lambda})(x, a) {\bm h}_\theta(a|x) \mu(dx)-\frac{\delta \mathscr{H}^\gamma}{\delta {\bm h}}(t, \mu, {\bm h}_{\theta, \lambda})(x, a){\bm h}^*(a|x)\mu(dx) \Big)d \lambda\\
=& \lim_{\theta \to 0} \frac{1}{\theta} \int_0^1 \Big(\E^e\Big[\frac{\delta \mathscr{H}^\gamma}{\delta {\bm h}}\big(t, \mu, {\bm h}_{\theta, \lambda})(\xi, \phi_{{\bm h}_\theta}(\xi, U)\big) - \frac{\delta \mathscr{H}^\gamma}{\delta {\bm h}}\big(t, \mu, {\bm h}_{\theta, \lambda})(\xi, \phi_{{\bm h}^*}(\xi, U) \big)\Big]\Big)d \lambda\\
=&\E^e\Big[\nabla_a \frac{\delta \mathscr{H}^\gamma}{\delta {\bm h}}\big(t, \mu, {\bm h}^* \big)(\xi, \phi_{{\bm h}^*}(\xi, U))\trans \big(\Tc \circ \phi_{{\bm h}^*}(\xi, U) - \phi_{{\bm h}^*}(\xi, U)\big)\Big]\\
=& \int_{\R^d \times \Ac} \nabla_a \frac{\delta \mathscr{H}^\gamma}{\delta {\bm h}}\big(t, \mu, {\bm h}^*\big)(x, a)\trans \big(\Tc(a) - a\big){\bm h}^*(a|x)\mu(dx) =0,
\end{align*}
where the second equality follows from the definition of partial linear functional derivative and ${\bm h}_{\theta, \lambda} = (1 - \lambda) {\bm h} + \lambda {\bm h}_\theta$.
By the arbitrariness of the map $\Tc$, we deduce that
\begin{align*}
\nabla_a  \frac{\delta \mathscr{H}^\gamma}{\delta {\bm h}}(t,  \mu, {\bm h}^*; {\bm \pi})(x, a) = \nabla_a \frac{\delta \mathscr{H}}{\delta {\bm h}}(t, \mu, {\bm h}^*; {\bm \pi})(x, a) -\gamma \nabla_a \log {\bm h}^*(a|x) = 0,
\end{align*}
which yields the desired result.
\end{proof}

\begin{Remark}
In particular, when there is no common noise, i.e., $\sigma_o = 0$, or when the common noise is uncontrolled, i.e., $\sigma_o = \sigma_o(t, x, \mu)$, the fixed point of \eqref{equ:mapI} reduces to the conventional Gibbs measure characterization, which is consistent with (2.11) in \cite{weiyu2023}.
When there is no common noise nor mean-field term, \eqref{equ:mapI} becomes (13) in \cite{jiazhou2022}.
\end{Remark}

Based on Theorem \ref{lemma:first-order}, the learning procedure starts with some policy ${\bm \pi}$ and produces a new policy ${\bm \pi}'$ that improves $\mathscr{H}^\gamma(t, \mu, {\bm h}; {\bm \pi})$. The next result shows that the resulting new policy that improves $\mathscr{H}^\gamma(t, \mu, {\bm h}; {\bm \pi})$ will also improve the value function,  and if the iterated new policy cannot improve the value function any more,  it must be an optimal policy.

\begin{Theorem}[Policy improvement]\label{thm:policy_improvement}
For a given ${\bm \pi} \in \Pi$, select a new policy ${\bm \pi}'$ such that $\mathscr{H}^\gamma(s, \mu, {\bm \pi}'; {\bm \pi}) \geq \mathscr{H}^\gamma(s, \mu, {\bm \pi}; {\bm \pi})$  holds for any $ s\in [t, T]$, we then have $J(t, \mu; {\bm \pi}') \geq J(t, \mu; {\bm \pi})$.
\end{Theorem}
\begin{proof}
For two given admissible policies ${\bm \pi}, {\bm \pi}' \in \Pi$, and any $0 \leq t \leq T$,  by applying It\^o's formula in \cite{CarD2} to $e^{-\beta(s-t)}\tilde J(s, \mu_s^{{\bm \pi}'}; {\bm \pi})$ between $t$ and $T$, we get that
{\small
\begin{align}
&\E^e\biggl[e^{-\beta (T-t)} \tilde J(T, \mu_T^{{\bm \pi}'}; {\bm \pi}) - \tilde J(t, \mu_t^{{\bm \pi}'}; {\bm \pi}) + \int_{t}^T  e^{-\beta (s-t)}\Big(r(s, \tilde X_s^{{\bm \pi}'}, \mu_s^{{\bm \pi}'}, a_s^{{\bm \pi}'}) + \gamma \mathcal{E}(s, \mu_s^{{\bm \pi}'}, {\bm \pi}') \Big) ds\biggl] \nonumber\\
= & \E^e\biggl[\int_t^T e^{-\beta (s-t)} \biggl(\frac{\partial \tilde J}{\partial t}(s, \mu_s^{{\bm \pi}'}; {\bm \pi}) - \beta \tilde J(s, \mu_s^{{\bm \pi}'}; {\bm \pi}) + \mathscr{H}^\gamma(s, \mu_s^{{\bm \pi}'}, {\bm \pi}'; {\bm \pi})\biggl)ds \biggl]. \nonumber
\end{align}}
Using $\mu_t^{{\bm \pi}'} = \mu$ and $\tilde J(T, \mu; {\bm \pi}) = \hat g(\mu)$, we rewrite the above equality as
\begin{align}\label{proof:policy-improvement-1}
 \tilde J(t, \mu; {\bm \pi}')-  \tilde J(t, \mu; {\bm \pi}) =  \E^e\biggl[\int_t^T e^{-\beta (s -t)} \biggl(\frac{\partial \tilde J}{\partial t}(s, \mu_s^{{\bm \pi}'}; {\bm \pi}) - \beta \tilde J(s, \mu_s^{{\bm \pi}'}; {\bm \pi}) + \mathscr{H}^\gamma(s, \mu_s^{{\bm \pi}'}, {\bm \pi}'; {\bm \pi})\biggl)ds\biggl].
\end{align}
Therefore, for any $(s, \mu) \in [t, T] \times \Pc_2(\R^d)$, $\mathscr{H}^\gamma(s, \mu, {\bm \pi}'; {\bm \pi}) \geq \mathscr{H}^\gamma(s, \mu, {\bm \pi}; {\bm \pi})$, we obtain that
\begin{align*}
& \tilde J(t, \mu; {\bm \pi}')-  \tilde J(t, \mu; {\bm \pi})\\
= & \E^e\biggl[\int_t^T e^{-\beta (s-t)} \biggl(\frac{\partial \tilde J}{\partial t}(s, \mu_s^{{\bm \pi}'}; {\bm \pi}) - \beta \tilde J(s, \mu_s^{{\bm \pi}'}; {\bm \pi}) + \mathscr{H}^\gamma(s, \mu_s^{{\bm \pi}'}, {\bm \pi}'; {\bm \pi})\biggl)ds\biggl] \nonumber\\
\geq & \E^e\biggl[\int_t^T e^{-\beta (s-t)} \biggl(\frac{\partial \tilde J}{\partial t}(s, \mu_s^{{\bm \pi}'}; {\bm \pi}) - \beta \tilde J(s, \mu_s^{{\bm \pi}'}; {\bm \pi}) + \mathscr{H}^\gamma(s, \mu_s^{{\bm \pi}'}, {\bm \pi}; {\bm \pi})\biggl)ds\biggl]= 0,
\end{align*}
where the last equality holds because of the dynamic programming equation \eqref{equ:dynamic-programming-equation}.
\end{proof}

\begin{Corollary}\label{cor:policy-improvement}
For a given ${\bm \pi} \in \Pi$, define ${\bm \pi}' = \mathcal{I}({\bm \pi})$, with $\mathcal{I}$ given in Theorem \ref{lemma:first-order}.
Then $J(t, \mu; {\bm \pi}') \geq J(t, \mu; {\bm \pi})$.
Conversely, if there exists some $\hat {\bm \pi} \in \Pi$ such that $J(t, \mu; \hat {\bm \pi}') = J(t, \mu; \hat {\bm \pi})$ for any $(t, \mu) \in [0, T] \times \Pc_2(\R^d)$, with $\hat {\bm \pi}' = \Ic(\hat {\bm \pi})$, then $\hat {\bm \pi}$ is an optimal policy of \eqref{equ:optimal_value_function}.
\end{Corollary}

\begin{proof}
If we take ${\bm \pi}' = \mathcal{I}({\bm \pi})$, then for any $(t, \mu) \in [0, T] \times \Pc_2(\R^d)$, it holds that $\mathscr{H}^\gamma(s, \mu, {\bm \pi}'; {\bm \pi}) \geq \mathscr{H}^\gamma(s, \mu, {\bm \pi}; {\bm \pi})$. By Theorem \ref{thm:policy_improvement}, we deduce that  $\tilde J(t, \mu; {\bm \pi}')\geq \tilde J(t, \mu; {\bm \pi})$.

We next prove the second claim. By \eqref{proof:policy-improvement-1} and $\tilde J(t, \mu; \hat {\bm \pi}) = \tilde J(t, \mu; \hat {\bm \pi}')$, we get that for any $(t, \mu) \in [0, T] \times \Pc_2(\R^d)$,
{\small
\begin{align}
\E^e\biggl[\int_t^T e^{-\beta (s -t)} \biggl(\frac{\partial \tilde J}{\partial t}(s, \mu_s^{{\bm \pi}'}; \hat {\bm \pi}) - \beta \tilde J(s, \mu_s^{{\bm \pi}'}; {\bm \pi}) + \mathscr{H}^\gamma(s, \mu_s^{{\bm \pi}'}, \hat {\bm \pi}'; \hat {\bm \pi})\biggl)ds\biggl] =0. \label{proof:policy-improvement-2}
\end{align}}Similarly, in view of \eqref{proof:policy-improvement-1} and $\tilde J(t + h,  \mu_{t + h}^{\hat{\bm \pi}'}; \hat {\bm \pi}) = \tilde J(t + h, \mu_{t + h}^{\hat{\bm \pi}'}; \mathcal{I}(\hat{\bm \pi}))$, $\P^0$-a.s. for any $h \in [0, T-t)$, we have that
{\small
\begin{align}
\E^e\biggl[\int_{t + h}^T e^{-\beta (s -t)} \biggl(\frac{\partial \tilde J}{\partial t}(s, \mu_s^{\hat {\bm \pi}'}; \hat{\bm \pi}) - \beta J(s, \mu_s^{\hat {\bm \pi}'}; \hat{\bm \pi}) + \mathscr{H}^\gamma(s, \mu_s^{\hat {\bm \pi}'}, \hat {\bm \pi}'; \hat {\bm \pi}) \biggl)ds\biggl] = 0. \label{proof:policy-improvement-3}
\end{align}}Subtracting \eqref{proof:policy-improvement-3} from \eqref{proof:policy-improvement-2} and dividing $h$ on both sides, we get that
{\small
\begin{align*}
 \frac{1}{h} \E^e\biggl[\int_t^{t + h} e^{-\beta (s -t)} \biggl(\frac{\partial J}{\partial t}(s, \mu_s^{\hat {\bm \pi}'}; \hat {\bm \pi}) - \beta J(s, \mu_s^{\hat {\bm \pi}'}; \hat {\bm \pi}) + \mathscr{H}^\gamma(s, \mu_s^{\hat {\bm \pi}'}, \hat {\bm \pi}'; \hat {\bm \pi})\biggl)ds\biggl]
= 0.
\end{align*}}
By the continuity of $b, \sigma$, $r$, $\tilde J$ and $\mu_s^{\hat {\bm \pi}'}$, it holds by sending $h \to 0$ that
\begin{align}
\frac{\partial \tilde J}{\partial t}(t, \mu; \hat {\bm \pi}) - \beta \tilde J(t, \mu; \hat {\bm \pi}) + \mathscr{H}^\gamma(t, \mu, \hat {\bm \pi}'; \hat {\bm \pi})=0.
\end{align}
As $\tilde J(t, \mu; \hat {\bm \pi})$ satisfies the equation \eqref{equ:dynamic-programming-equation},
we arrive at  $\mathscr{H}^\gamma(t, \mu, \hat {\bm \pi}'; \hat {\bm \pi})= \mathscr{H}^\gamma(t, \mu, \hat {\bm \pi}; \hat {\bm \pi})$.
Recall from Proposition \ref{lemma:first-order} that $\hat {\bm \pi}' = \mathcal{I}(\hat {\bm \pi})$ is the unique maximizer of $\mathscr{H}^\gamma(t, \mu, {\bm h}; \hat {\bm \pi})$, then we have $\mathcal{I}(\hat {\bm \pi}) = \hat {\bm \pi}$. By a standard verification argument for the entropy regularized MFC problem, it holds that $J(t, \mu, \hat {\bm \pi}) = J^*(t, \mu)$ and hence $\hat {\bm \pi}$ is an optimal policy.
\end{proof}

\section{Continuous-Time Integrated q-Function}\label{sec:q-function}

We investigate in this section the proper definition of the continuous-time integrated q-function (Iq-function), which lays the theoretical foundation of q-learning theory for MFC problems. 

Similar to \cite{weiyu2023}, let us consider a ``perturbed policy" $\bar {\bm \pi} \in \Pi$, which takes ${\bm h} \in \Pc_{ac}(\Ac|\R^d) $ on $[t, t+ \Delta t)$, and then ${\bm \pi} \in \Pi$ on $[t+\Delta t, T)$. Then $X_s^{\Dc, \bar{\bm \pi}}$ on $[t, T)$ is governed by
\begin{align*}
 dX_s^{\Dc, \bar {\bm \pi}} &= b(s, X_s^{\Dc, \bar {\bm \pi}},  \mu_s^{\Dc, \bar {\bm \pi}},  {a}_{\delta(s)}^{\bm h})ds + \sigma(s, X_s^{t, \xi, \bar {\bm \pi}}, \mu_s^{\Dc, \bar {\bm \pi}}, {a}_{\delta(s)}^{\bm h}) dW_s, \; s \in [t, t+\Delta t), \\
 & \;\;\;+ \sigma_{o}(s, X_s^{t, \xi, \bar {\bm \pi}}, \mu_s^{\Dc, \bar {\bm \pi}}, {a}_{\delta(s)}^{\bm h}) dB_s,\; X_t^{\Dc, \bar{\bm \pi}} = \xi,\\
 dX_s^{\Dc, \bar {\bm \pi}} &= b(s, X_s^{\Dc, \bar {\bm \pi}},\mu_s^{\Dc, \bar {\bm \pi}}, {a}_{\delta(s)}^{\bm \pi})ds + \sigma(s, X_s^{t, \xi, \bar {\bm \pi}}, \mu_s^{\Dc, \bar {\bm \pi}}, {a}_{\delta(s)}^{\bm \pi}) dW_s, \; s \in [t+\Delta t, T), \\
 & \;\;\; + \sigma_{o}(s, X_s^{t, \xi, \bar {\bm \pi}},  \mu_s^{\Dc, \bar {\bm \pi}}, {a}_{\delta(s)}^{\bm \pi}) dB_s, \; X_{t+ \Delta t}^{\Dc, \bar {\bm \pi}} = X_{t+\Delta t}^{\Dc, {\bm h}}.
\end{align*}
We first consider the discrete time IQ-function defined on $[0, T] \times L^2(\Omega; \R^d) \times \Pc_{ac}(\Ac|\R^d)$ {\it independent of discretely sampling},  with the fixed time interval $\Delta t$ and the entropy regularizer that
\begin{align}
& Q_{\Delta t} (t, \xi, {\bm h}; {\bm \pi}) = :\lim_{|\Dc| \to 0} Q_{\Delta t}^{\Dc}(t, \xi, {\bm h}; {\bm \pi})\nonumber\\
=& \lim_{|\Dc| \to 0} \E^e\biggl[\int_t^{t + \Delta t} e^{-\beta(s-t)}\Big(r(s,  X_s^{\Dc, \bar {\bm \pi}},  \mu_s^{\Dc, \bar {\bm \pi}}, a^{\bm h}_{\delta(s)}) +\gamma E_{\bm h}(\delta(s), X_{\Dc, \delta(s)}^{\bar{\bm \pi}}, \mu_s^{\Dc, \bar{\bm \pi}})\Big)ds\nonumber\\
& \;+ \int_{t + \Delta t}^T  e^{-\beta(s-t)} \Big(r(s,  X_s^{\Dc, \bar {\bm \pi}}, \mu_s^{\Dc, \bar{\bm \pi}}, a^{\bm \pi}_{\delta(s)}) + \gamma E_{\bm \pi}(\delta(s), X_{\Dc, \delta(s)}^{\bar{\bm \pi}}, \mu_s^{\Dc, \bar{\bm \pi}})\Big) ds \nonumber\\
& \;+ e^{-\beta(T-t)} g(X_T^{\Dc, \bar {\bm \pi}},  \mu_T^{\Dc, \bar {\bm \pi}}) \Big| X_t^{\Dc, \bar{\bm \pi}} = \xi\biggl]. \nonumber
\end{align}
By noting that $Q_{\Delta t}^{\Dc}(t, \xi, {\bm h}; {\bm \pi}) = J^{\Dc}(t, \xi; \bar{\bm \pi})$ and the equivalence result in Proposition \ref{prop:approximation}, we have that
\begin{align*}
Q_{\Delta t} (t, \xi, {\bm h}; {\bm \pi}) = \lim_{|\Delta \Dc| \to 0} J^{\Dc}(t, \xi; \bar{\bm \pi}) = \tilde J(t, \mu; \bar{\bm \pi}).
\end{align*}
 Consequently, we can rewrite $Q_{\Delta t}$ in terms of the relaxed control formulation
\begin{align}
 Q_{\Delta}(t, \xi, {\bm h}; {\bm \pi}) &= \E^e\biggl[\int_t^{t + \Delta t} e^{-\beta(s-t)}\Big(\hat r_{\bm h}(s, \mu_s^{\bar{\bm \pi}})+\gamma \mathcal{E}(s, \mu_s^{\bar{\bm \pi}}, {\bm h})\Big)ds\nonumber\\
& \;+ \int_{t + \Delta t}^T  e^{-\beta(s-t)} \Big(\hat r_{\bm \pi}(s, \mu_s^{\bar{\bm \pi}}) + \gamma \mathcal{E}(s, \mu_s^{\bar{\bm \pi}}, {\bm \pi})\Big) ds + e^{-\beta(T-t)} \hat g(\mu_T^{\bar{\bm \pi}})\biggl]. \nonumber
\end{align}
Noting the collapse of  IQ-function to the value function as $\Delta t \to 0$, we instead consider the first order derivative of $Q_{\Delta t}$ by using the flow property of $\mu_s^{\bar{\bm \pi}}$ and
applying It\^o's formula (see Theorem 4.14 in \cite{CarD2}) to $e^{-\beta  s} \tilde J(s, \mu_s^{{\bm h}}; {\bm \pi})$ between $t$ and $t + \Delta t$ that
 \begin{align}
  Q_{\Delta t} (t, \xi, {\bm h}; {\bm \pi}) =&  \tilde J(t, \mu; {\bm \pi}) + 
 \E^e\biggl[\int_t^{t + \Delta t} e^{-\beta(s-t)}\Big(\frac{\partial \tilde J}{\partial t} (s, \mu_s^{\bm h}; {\bm \pi}) - \beta \tilde J(s, \mu_s^{\bm h}; {\bm \pi}) +\mathscr{H}^\gamma(s, \mu_s^{\bm h}, {\bm h}; {\bm \pi})\Big)ds \biggl], \nonumber
 \end{align}
 where $\mathscr{H}^{\gamma}$ is defined in \eqref{equ:F}.
 By the continuity of $\tilde J$ and $\mu_s^{\bm h}$ with respect to $s$, it holds that
\begin{align}
  Q_{\Delta t} (t, \xi, {\bm h}; {\bm \pi}) \approx & \tilde J(t, \mu; {\bm \pi}) + \Delta t \Big( \frac{\partial \tilde J}{\partial t}(t, \mu; {\bm \pi}) - \beta \tilde J(t, \mu; {\bm \pi})  + \mathscr{H}^\gamma(t, \mu, {\bm h}; {\bm \pi})\Big) + o(\Delta t). \label{equ:def_Q}
\end{align}
This leads to the next definition of continuous-time Iq-function.

\begin{Definition}\label{def:coupled-q-function}
Given a  policy ${\bm \pi} \in \Pi$, for any $(t, \mu, {\bm h}) \in [0, T] \times \Pc_2(\R^d) \times \Pc_{ac}(\Ac|\R^d)$, we define the continuous-time integrated q-function (Iq-function) by
\begin{align*}
q^{\gamma}(t, \mu, {\bm h}; {\bm \pi}) & := \lim_{\Delta t \to 0} \frac{ Q_{\Delta t} (t, \xi, {\bm h}; {\bm \pi}) - \tilde J(t, \mu; {\bm \pi})}{\Delta t} =  \frac{\partial \tilde J}{\partial t}(t, \mu; {\bm \pi}) - \beta \tilde J(t, \mu; {\bm \pi})  + \mathscr{H}^\gamma(t, \mu, {\bm h}; {\bm \pi}).
\end{align*}
We also call $q^0(t, \mu, {\bm h}; {\bm \pi}) = q^\gamma(t, \mu, {\bm h}; {\bm \pi}) - \gamma \mathcal{E}(t, \mu, {\bm h})$ the unregularized Iq-function.
\end{Definition}

It is straightforward to see that $q^{\gamma}$ (resp. $q^0$) equals to $\mathscr{H}^\gamma$ in \eqref{equ:regularized_minimization} (resp. $\mathscr{H}$ in
 \eqref{equ:F}) compensated by the dispersion term $\frac{\partial \tilde J}{\partial t}(t, \mu; {\bm \pi}) - \beta \tilde J(t, \mu; {\bm \pi})$. Therefore,
 we obtain the following corollary, which is an immediate consequence of Theorem \ref{lemma:first-order}.
 \begin{Corollary}\label{corollary:optimal-policy}
 Let Assumptions \ref{ass:b-sigma-sigmao} and \ref{assum:2nd-mu-derivative-J} hold.  Given ${\bm \pi} \in \Pi$ and $(t, \mu) \in [0, T] \times \Pc_{2 + \delta}(\R^d)$ for some $\delta >0$, there exists a unique maximizer to
$${\bm h}^* = \argmax_{{\bm h} \in \Pc_{ac}(\Ac|\R^d)} q^{\gamma}(t, \mu, {\bm h}; {\bm \pi})$$
if and only if
 \begin{align*}
{\bm h}^*(a|t, x, \mu) = \Phi_{\bm \pi}({\bm h}^*)= \frac{\exp\Big\{\frac{1}{\gamma} \frac{\delta q^0}{\delta {\bm h}}(t, \mu, {\bm h}^*; {\bm \pi})(x, a)\Big\}}{\int_{\Ac} \exp\Big\{\frac{1}{\gamma} \frac{\delta q^0}{\delta {\bm h}}(t, \mu, {\bm h}^*; {\bm \pi})(x, a)\Big\}da}.
\end{align*}
Furthermore, if there exists some ${\bm \pi}^* \in \Pi$ satisfying the two-layer fixed point to
$${\bm \pi}^* = \argmax_{{\bm h} \in \Pc_{ac}(\Ac|\R^d)} q^{\gamma,*}(t, \mu, {\bm h})$$
or equivalently the two-layer fixed point to
\begin{align}\label{twofix}
{\bm \pi}^*(a|t, x, \mu) = \frac{\exp\Big\{\frac{1}{\gamma} \frac{\delta q^{0, *}}{\delta {\bm h}}(t, \mu, {\bm \pi}^*)(x, a)\Big\}}{\int_{\Ac} \exp\Big\{\frac{1}{\gamma} \frac{\delta q^{0, *}}{\delta {\bm h}}(t, \mu, {\bm \pi}^*)(x, a)\Big\}da},
\end{align}
where $q^{\gamma, *}(t, \mu, {\bm h}): = q^{\gamma}(t, \mu, {\bm h}; {\bm \pi}^*)$ and  $q^{0, *}(t, \mu, {\bm h}): = q^{0}(t, \mu, {\bm h}; {\bm \pi}^*)$,
then ${\bm \pi}^*$ is an optimal policy.
 \end{Corollary}

\begin{Remark}
\begin{figure}[h]
\[
\xymatrix{
\bm{\pi}^0\ar[r]^-{\mathcal{I}} & \cdots\ar[r]^-{\mathcal{I}} & \bm{\pi}^n\ar[d]_-{\Phi_{\bm{\pi}^n}}\ar[rrrr]^-{\mathcal{I}} &&&& \bm{\pi}^{n+1}\ar[r]^-{\mathcal{I}} & \cdots\ar[r]^-{\mathcal{I}} & \bm{\pi}^*\\
&&\bm{\pi}^{n,1}\ar[r]^-{\Phi_{\bm{\pi}^n}} & \cdots\ar[r]^-{\Phi_{\bm{\pi}^n}} & \bm{\pi}^{n,\ell}\ar[r]^-{\Phi_{\bm{\pi}^n}} & \bm{\pi}^{n,\ell+1}\ar[r]^-{\Phi_{\bm{\pi}^n}} & \cdots\ar[u]_-{\Phi_{\bm{\pi}^n}}
}
\]
\caption{Illustration of two-layer fixed point \eqref{twofix}}\label{figure:twofix}
\end{figure}
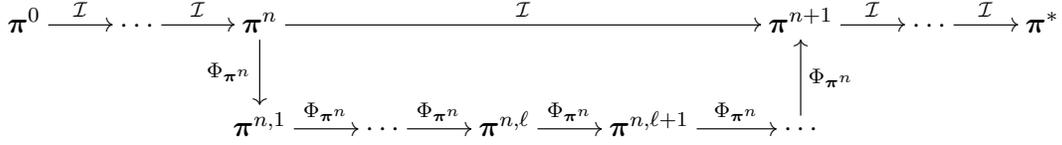

We could search for the optimal policy by considering \eqref{twofix} as a two-layer fixed point problem. Specifically, starting with an initial policy ${\bm \pi}^0$, at each iteration  $n \in \N$,  we derive ${\bm \pi}^{n+1}$ by looking for the fixed point of the map $\Phi_{{\bm \pi }^n}$, which is the inner layer of the two-layer fixed point problem. Recall that $\Ic$ is a map from ${\bm \pi}^n$ to ${\bm \pi}^{n+1}$, that is, ${\bm \pi}^{n+1} = \Ic({\bm \pi}^n)$. And then the optimal policy ${\bm \pi}^*$ is a fixed point of the map $\Ic$, which is the outer layer of the two-layer fixed point problem. See Figure \ref{figure:twofix} for the illustration.
\end{Remark}

\section{Linear Quadratic MFC and Gaussian Optimal Policy}\label{sec:LQ}
Let us consider a controlled  linear McKean-Vlasov dynamics with $\Ac = \R^p$ and assume $n=m=1$ for simplicity. The coefficients of the dynamics are given by
\begin{align}
b(t, x, \mu, a) &= b_0(t) + B(t) x + \bar B(t) \bar\mu + C(t) a, \nonumber\\
\sigma(t, x, \mu, a) &= \vartheta(t) + D(t) x + \bar D(t) \bar\mu + F(t)a, \label{LQdynamic}\\
\sigma_o(t, x, \mu, a) &= \vartheta_o(t) + D_o(t) x + \bar D_o(t) \bar\mu + F_o(t)a. \nonumber
\end{align}
The running and terminal reward functions are given by
\begin{align}
r(t, x, \mu, a) = x\trans M(t) x + \bar\mu\trans \bar M(t) \bar\mu + a \trans R(t) a +x\trans O(t), \;\;g(x, \mu) = x \trans P x + \bar\mu\trans \bar P\bar\mu\label{LQreward}
\end{align}
Here, $b_0(t), \vartheta(t)$, $\vartheta_o(t)$ and $O(t)$ are deterministic functions of $t$ valued in $\R^d$, $B(t)$, $\bar B(t)$, $D(t)$, $\bar D(t)$, $D_o(t)$, $\bar D_o(t)$, $M(t)$ and $\bar M(t)$ are deterministic functions of $t$ valued in $\R^{d \times d}$, $C(t), F(t), F_o(t)$ are deterministic functions of $t$ valued in $\R^{d \times p}$, $R(t)$ is deterministic matrix functions of $t$ valued in $\R^{p \times p}$, and $P$ and $\bar P$ are constant matrices in $\R^{d \times d}$. We may assume without loss of generality that $M(t), \bar M(t), R(t), P, \bar P$ are symmetric matrices and $\beta =0$.

Denote the mean and variance of $\mu$ by $\bar\mu = \int_{\R^d} x \mu(dx)$ and ${\rm Var}(\mu)(\Lambda) = \int_{\R^d} (x - \bar\mu)\trans \Lambda (x -\bar\mu) \mu(dx)$, respectively, for a symmetric matrix $\Lambda \in \R^{d \times d}$. We obtain explicit expressions of the optimal value function $\tilde J^*$ and the optimal policy ${\bm \pi}^*$ in the next result.
\begin{Theorem}\label{thm:LQ}Under the condition
\begin{align*}
({\bf H}) \;\; P \preceq 0, \; P + \bar P \preceq 0, \;M(t) \preceq 0,\; M(t) + \bar M(t) \preceq 0,\; R(t) \preceq -\delta I_q,
\end{align*}
for some $\delta >0$,
the optimal value function $\tilde J^*$ takes the quadratic form
\begin{align}\label{finite-LQ-J*-formulation}
\tilde J^*(t, \mu) = {\rm Var}(\mu)(\Lambda^*(t)) + \bar\mu\trans \Gamma^*(t) \bar\mu + \bar\mu\trans \zeta^*(t) + \chi^*(t),
\end{align}
and the optimal policy is unique and satisfies the Gaussian type that
\begin{align}\label{LQ-optimal-policy}
{\bm \pi}^*(\cdot|t, x, \mu) = \Nc\Big(- (U_t^*)^{ -1} S^*_t (x - \bar\mu) -(V^*_t)^{-1} Z^*_t \bar\mu - \frac{1}{2}(V_t^*)^{-1} Y_t^*, - \frac{\gamma}{2} (U^*_t)^{-1}\Big).
\end{align}
Here, we set $U_t= U(t, \Lambda(t))$, $V_t = V(t, \Gamma(t))$, $S_t = S(t, \Lambda(t))$, $Z_t = Z(t, \Gamma(t), \Lambda(t))$, $Y_t = Y(\zeta(t), \Lambda(t),\Gamma(t))$ such that
\begin{align*}
\left\{
\begin{array}{lll}
U_t &= F\trans \Lambda(t) F +  F_o \trans \Lambda(t) F_o + R,\\
V_t & = F\trans \Lambda(t) F  + F_o \trans \Gamma(t) F_o+ R,\\
S_t & = C\trans \Lambda(t) + F\trans \Lambda(t) D + F_o\trans \Lambda(t) D_o,\\
Z_t & = C\trans\Gamma(t) + F\trans \Lambda(t)\big(D + \bar D\big) + F_o \trans \Gamma(t)\big( D_o+ \bar D_o\big),\\
Y_t & = C\trans \zeta(t) + 2F\trans \Lambda(t) \vartheta + 2F_o \trans \Gamma(t) \vartheta_0,
\end{array}
\right.
\end{align*}
and $\Lambda^*(t)$, $\Gamma^*(t)$, $\zeta^*(t)$ and $\chi^*(t)$ satisfy
\begin{align}\label{finite-ODELambda}
\left\{
\begin{array}{rcl}
 (\Lambda^*)'(t)  + M + D\trans \Lambda^*(t) D + D_o\trans \Lambda^*(t) D_o + B\trans \Lambda^*(t) + \Lambda^*(t) B - (S_t^*)\trans (U_t^*)^{-1} S_t^*=0,\\
 \Lambda^*(T) = P,
 \end{array}
\right.
\end{align}
\begin{align}\label{finite-ODEGamma}
\left\{
\begin{array}{rcl}
(\Gamma^*)'(t) + M+ \bar M + \big(D +\bar D)\trans \Lambda^*(t)\big(D + \bar D) +  (\bar D_o + D_o)\trans \Gamma^*(t)(\bar D_o + D_o)\\
  + (B + \bar B)\trans \Gamma^*(t)+ \Gamma^*(t) (B + \bar B)- (Z_t^*)\trans (V_t^*)^{-1}Z_t^* =0,\\
\Gamma^*(T) = P + \bar P.
\end{array}
\right.
\end{align}

\begin{align}\label{finite-ODEzeta}
\left\{
\begin{array}{rcl}
(\zeta^*)'(t)  +  (B + \bar B)\trans \zeta^*(t) +  2 \Gamma^*(t) b_0 + 2(D + \bar D)\trans \Lambda^*(t) \vartheta + 2(\bar D_o + D_o)\trans \Gamma^*(t)  \vartheta_o\\
-(Z_t^*)\trans (V_t^*)^{-1} Y_t^* + O=0,\\
 \zeta^*(T) =0,
\end{array}
\right.
\end{align}

\begin{align}\label{finite-ODEchi}
\left\{
\begin{array}{rcl}
 (\chi^*)'(t) + \vartheta\trans \Lambda^*(t)\vartheta +  \vartheta_o\trans \Gamma^*(t) \vartheta_o + b_0\trans \zeta^*(t) - \frac{1}{4}(Y_t^*)\trans (V_t^*)^{-1} Y_t^*\\
 + \frac{\gamma}{2}\log \Big((-\gamma\pi)^{p} {\rm det}((U_t^*)^{-1})\Big)  =0,\\
\chi^*(T) = 0.
\end{array}
\right.
\end{align}
\end{Theorem}

\begin{Remark} In view of \eqref{finite-ODELambda}-\eqref{finite-ODEchi}, when $\gamma$ tends to zero,  the solution $(\Lambda, \Gamma, \zeta, \chi)$ of the LQ-MFC problem with entropy regularizer reduces to that of the classical LQ-MFC problem, and the relaxed optimal policy  reduces to the optimal strict control in \cite{phamwei2017}.
\end{Remark}

\begin{proof}[Proof of Theorem \ref{thm:LQ}]
We conjecture that $J^*$ takes the following quadratic form in \eqref{finite-LQ-J*-formulation}.
 One can easily check that $J^* \in \Cc^{1, 2}([0, T] \times \Pc_2(\R^d))$ with
\begin{align*}
\frac{\partial J^*}{\partial t}(t, \mu) & = {\rm Var}(\mu)( (\Lambda^*)'(t)) + \bar\mu\trans (\Gamma^*)'(t) \bar\mu + \bar\mu\trans(\zeta^*)'(t) + (\chi^*)'(t),\\
\partial_\mu J^*(t, \mu)(x) & = 2 \Lambda^*(t) (x - \bar\mu) + 2 \Gamma^*(t) \bar \mu + \zeta^*(t),\\
\partial_x\partial_\mu J^*(t, \mu)(x) & = 2\Lambda^*(t),\; \partial_\mu^2 J^*(t, \mu)(x, x') = 2(\Gamma^*(t) - \Lambda^*(t)).
\end{align*}
For simplicity, we will suppress the time variable $t$ in the rest of the proof. Plugging \eqref{finite-LQ-J*-formulation} in the first order condition \eqref{equ:first-order}, together with \eqref{equ:derivativeF}, we have that
\begin{align}
&\big(b_0 + B x + \bar B \bar\mu + C a\big)\trans \big(2\Lambda^*(x - \bar\mu) + 2 \Gamma^* \bar\mu + \zeta\big) + \big(\vartheta + D x + \bar D \bar\mu + F a\big) \trans \Lambda^* \big(\vartheta + D x + \bar D \bar\mu + F a\big) \nonumber\\
& + \big(\vartheta_o+ D_o x + \bar D_o \bar\mu + F_o a\big)\trans \Lambda^* \big(\vartheta_o + D_o x + \bar D_o \bar\mu + F_oa\big)+ x\trans M x + \bar\mu\trans \bar M \bar\mu + a \trans R a \nonumber\\
& +\big(\vartheta_o + D_o x + \bar D_o \bar\mu + F_o a\big) \trans \int_{\R^d} \int_{\R^p} 2(\Gamma^* - \Lambda^* ) \big(\vartheta_o + D_o x' + \bar D_o \bar\mu + F_o a\big){\bm \pi}^*(a|t, x', \mu)da\mu(dx') \nonumber\\
= &\gamma \log {\bm \pi}^*(a|t, x, \mu) + \kappa(t, x,\mu). \label{equ:LQ-first-order-condition}
\end{align}
By comparing both sides of the above equality, the fixed point satisfies the form of
\begin{align*}
\log {\bm \pi}^*(a|t, x, \mu) = -\frac{1}{2}\big(a - m(t, x, \mu)\big)\trans \Sigma^{-1}(t) \big(a - m(t, x, \mu)\big) - \frac{1}{2}\log \Big((2\pi)^{p} {\rm det}(\Sigma(t))\Big),
\end{align*}
where $m(t, x, \mu) = K(t)(x -\bar\mu) + \bar K(t) \bar\mu + K_0(t)$.
Comparing the coefficients of terms $a$ and $a\trans (\cdot) a$ on both sides of the above equality, we have that
$\Sigma(t)$, $K(t)$, $\bar K(t)$ and $K_0(t)$ satisfy
\begin{align*}
\Sigma  &= -\frac{\gamma}{2} U_t^*, K = \frac{2\Sigma }{\gamma}S_t^*, \; \bar K  =  \frac{2\Sigma}{\gamma}\Big(Z_t^*+ F_o\trans (\Gamma^* -\Lambda^*) F_o \bar K\Big),\\
K_0 & = \frac{\Sigma}{\gamma}\Big(Y_t^* + 2F_o\trans(\Gamma^* - \Lambda^*)  F_0 K_0 \Big).
\end{align*}
Therefore, we verify that ${\bm \pi}^*$ is Gaussian in the form of \eqref{LQ-optimal-policy}.

After straightforward calculations,
we see that $J^*$ satisfies the dynamic programming equation \eqref{equ:dynamic-programming-equation} if and only if
\begin{align*}
&{\rm Var}(\mu)\Big((\Lambda^*)' + D\trans \Lambda^* D  + D_o\trans \Lambda^* D_o +  B\trans \Lambda^* + \Lambda^* B - \frac{\gamma}{2} K\trans \Sigma^{-1} K+ 2 (S_t^*)\trans K + M\Big)\\
& + \bar\mu\trans\Big((\Gamma^*)' + \big(D +\bar D)\trans \Lambda^*\big(D + \bar D) +  (D_o + \bar D_o)\trans \Gamma^* (D_o + \bar D_o) +  (B + \bar B)\trans \Gamma^* + \Gamma(B + \bar B)\\
& - \frac{\gamma}{2} \bar K \trans \Sigma^{-1} \bar K + \bar K\trans F_o\trans  (\Gamma^* - \Lambda^*)  F_o \bar K + 2(Z_t^*)\trans \bar K +  M + \bar M  \Big)\bar\mu \\
& + \bar\mu\trans \Big( \zeta' +  (B + \bar B)\trans \zeta  + 2 \Gamma^* b_0 + 2(D + \bar D)\trans \Lambda^* \vartheta + 2 (D_o  + \bar D_o)\trans \Lambda^* \vartheta_o  + 2 \big(FC + (D + \bar D)\trans \Lambda^* F\\
& + (D_o + \bar D_o)\trans \Gamma^* F_o\big) K_0  + \bar K\trans Y_t\trans  + \bar K\trans \big(-\gamma\Sigma^{-1} + F_o\trans (\Gamma^* - \Lambda^*) F_o \big) K_0 +{O}\Big) \\
& + (\chi^*)' + \vartheta\trans \Lambda^*\vartheta  + \vartheta_o\trans \Gamma^* \vartheta_o +  b_0\trans \zeta^* + \big(Y_t\big)\trans K_0 - \frac{\gamma}{2} K_0\trans \Sigma^{-1} K_0 + K_0\trans F_o \trans (\Gamma^* -\Lambda^*) F_o K_0 \\
&+ \frac{\gamma}{2}\log  \Big((2\pi)^{p} {\rm det}(\Sigma)\Big)  =0.
\end{align*}
Setting the coefficients of the terms ${\rm Var}(\mu)(\cdot)$, $\bar\mu\trans(\cdot)\bar\mu$, $\bar\mu$ to be zero, we arrive at the system of ODEs for $\Lambda^*(t), \Gamma^*(t), \zeta^*(t)$ and $\chi^*(t)$. It is known that, see e.g. \cite{Wonham1968}, \cite{Yong2013}, under the condition {\bf (H)},
for some $\delta >0$, the matrix Riccati equations \eqref{finite-ODELambda}-\eqref{finite-ODEGamma} admit the unique solution $(\Lambda^*, \Gamma^*)$ valued in symmetric negative definite matrices. Given the existence of $(\Lambda^*, \Gamma^*)$, we also have the existence of solution to the system of linear ODEs for $(\zeta^*, \chi^*)$.

%
%

Finally, we verify Assumption \ref{assum:2nd-mu-derivative-J} (ii). Denote
\begin{align*}
 \bar{\bm h} :=& \int_{\R^d}\int_{\R^ p} a {\bm h}(a|x)da \mu(dx) = \E^e_{\mu, {\bm h}}[a^{\bm h}],\\
{\rm Var}({\bm h})(\Lambda^*) :=& \int_{\R^d} \int_{\R^p} \big(a - \bar{\bm h}\big)\trans \Lambda^* \big(a - \bar{\bm h}\big) {\bm h}(a|x)da \mu(dx)
= \E^e_{\mu, {\bm h}}[(a^{\bm h}) \trans \Lambda^* a^{\bm h}] - \E^e_{\mu, {\bm h}}[a^{\bm h}]\trans \Lambda^* \E_{\mu, {\bm h}}[a^{\bm h}].
\end{align*}
By \eqref{equ:F}, we have that
\begin{align*}
\mathscr{H}(t, \mu, {\bm h}) &= {\rm Var}({\bm h})(U_t^*)+ \bar{\bm h} \trans V_t^* \bar{\bm h} +  \int_{\R^d}\int_{\R^p} a \trans \big((S_t^*)\trans (x - \bar\mu) + (Z_t^*)\trans \bar\mu + Y_t^*\big) {\bm h}(a|x)da\mu(dx)\\
& \;\;\;+ G(t, \mu),
\end{align*}
where $G$ is independent of ${\bm h}$.
For every $x \in \R^d$, we denote ${\bm h}_\theta(\cdot|x)$ as the law of the interpolated random variable $\phi_{{\bm h}^\theta}(x, U) = (1 - \theta) \phi_{{\bm h}_1}(x, U) + \theta \phi_{{\bm h}_0}(x, U)$. Thus $\mathscr{H}(t, \mu, {\bm h}_\theta)$ can be written in terms of $\phi_{{\bm h}^\theta}(x, U)$ that
\begin{align*}
\mathscr{H}(t, \mu, {\bm h}_\theta) &  = \E^e \big[\big(\phi_{{\bm h}^\theta}(\xi, U) - \E^e[\phi_{{\bm h}^\theta}(\xi, U)]\big)\trans U_t^* \big(\phi_{{\bm h}^\theta}(\xi, U) - \E^e[\phi_{{\bm h}^\theta}(\xi, U)]\big)\big]\\
&  + \E^e\big[\phi_{{\bm h}^\theta}(\xi, U)]\trans V_t^* \E^e[\phi_{{\bm h}^\theta}(\xi, U)\big]\\
& +  \E^e [\phi_{{\bm h}^\theta}(\xi, U)\trans \big((S_t^*)\trans (\xi - \bar\mu) + (Z_t^*)\trans \bar\mu + Y_t^*\big)\big]
+ G(t, \mu)\\
& \geq (1 - \theta) \mathscr{H}(t, \mu, {\bm h}_0) + \theta \mathscr{H}(t, \mu, {\bm h}_1),
\end{align*}
which implies that $\mathscr{H}(t, \mu, {\bm h})$ is displacement concave in view of $U_t^* \preceq 0$ and $V_t^* \preceq 0$.

By \cite{Villani09}, $\Ec(t, \mu, {\bm \pi})$ is also displacement concave. We thus conclude that $\mathscr{H}^\gamma(t, \mu, {\bm h})$ is displacement concave in ${\bm h}$
 and  Assumption \ref{assum:2nd-mu-derivative-J} (ii) holds.
\end{proof}

\vspace{0.2in}
\noindent
\textbf{Acknowledgement}:
X. Wei is supported by National Natural Science Foundation of China grant under no.12201343 and no.12571509.  X. Yu is supported by the Hong Kong RGC General Research Fund (GRF) under grant  no. 15211524.

\appendix

\section{Cooperative Mean-Field $N$-Agent Game}\label{appendix:N-player}
Here, we recall the cooperative $N$-agent game that is coordinated by a social planner in the continuous-time entropy regularized setting.

The state $X_t^{j, \Dc, \bm \pi}$ of agent $j \in \{1, \ldots, N\}$ satisfies the SDE
\begin{align}\label{agent-j-dynamics}
dX_s^{j, \Dc, \bm \pi} &= b(s, X_s^{j, \Dc, \bm \pi}, \mu_s^{N, \Dc, \bm \pi},  a_{\delta(s)}^{j, \Dc, \bm \pi}) ds + \sigma(s, X_s^{j, \Dc, \bm \pi}, \mu_s^{N, \Dc, \bm \pi}, a_{\delta(s)}^{j, \bm \pi}) dW_s^j \\
& \;\;\;\;\;+\sigma_o(s, X_s^{j, \Dc, \bm \pi}, \mu_s^{N, \Dc, \bm \pi}, a_{\delta(s)}^{j, \Dc, \bm \pi}) dB_s, \; X_t^{j, \Dc, \bm \pi} = x^j,\nonumber
\end{align}
where $W^1, \ldots, W^N$ are independent Brownian motions, and independent of $B$, and $\mu_s^{N, \Dc, \bm \pi} = \frac{1}{N} \sum_{j=1}^N\delta_{X_s^{j, \Dc, \bm \pi}}$ is the empirical measure, and $a_{\delta(s)}^{j, {\bm \pi}} \sim {\bm \pi}(\cdot|\delta(s), X_{\delta(s)}^{j, \Dc, \bm \pi}, \mu_s^{N, \Dc, \bm \pi})$ stands for the discretely sampled actions. The expected accumulated reward for the agent $j$ is
\begin{align*}
& \E^e\Big[\int_t^T e^{-\beta s} \big(r(X_s^{j, \Dc, \bm \pi}, \mu_s^{N, \Dc, \bm \pi}, a_{\delta(s)}^{j, \bm \pi}) +  \gamma E_{\bm \pi}(\delta(s), X_{j, \Dc, \delta(s)}^{\bm \pi}, \mu_{\delta(s)}^{N, \Dc, \bm \pi}) \big)ds+ g(X_T^{j, \Dc, \bm \pi}, \mu_T^{N, \Dc, \bm \pi})\Big].
\end{align*}
The learning procedure for the cooperative $N$-agent game is as follows. At each time $s \in [t, T]$, each agent $j$ observes the empirical measure $\mu_s^{N, \Dc, \bm \pi}$ and takes the action $a_{\delta(s)}^{j, \Dc, \bm \pi}$ according to the policy ${\bm \pi}: [0, T] \times \R^d \times \Pc_2(\R^d) \to \Pc(\Ac)$ assigned by the social planner. His state $X_s^{j, \Dc, \bm \pi}$ evolves according to \eqref{agent-j-dynamics} and he will receive an individual reward $r(X_s^{j, \Dc, \bm \pi}, \mu_s^{N, \Dc, \bm \pi}, a_{\delta(s)}^{j, \Dc, \bm \pi})$ via the interaction with the environment. 

At the social planner's level, she selects the policy ${\bm \pi}$, assigns it to all agents and observes the evolution of the empirical measure $\mu_s^{N, \Dc, \bm \pi}$ over time, and obtains the aggregated reward $\frac{1}{N} \sum_{j=1}^N r(X_s^{j, \Dc, \bm \pi}, \mu_s^{N, \Dc, \bm \pi}, a_{\delta(s)}^{j, \bm \pi})$ so as to ultimately maximize the overall aggregated reward. As $N \to + \infty$, this formulation leads to an exploratory MFC problem as described in Section \ref{sec:exploratory-form}.

\section{Heuristical Derivation of Relaxed Control Formulation}\label{sec:exploratory-average}
In this section, we discuss the corresponding relaxed formulation of MFC with controlled common noise. Let $\Cc_c^\infty(\R^d \times \R^m \times \R^n)$ denote the set of infinitely differentiable function $\phi: \R^d \times \R^m  \times \R^n \to \R$ with compact set, and let $D\phi$ and $D^2\phi$ denote gradient and Hessian of $\phi$, respectively. Let $\sigma_i$ and $\sigma_{o, i}$, $1 \leq i \leq d$, denote $i$-th row of $\sigma$ and $\sigma_o$, respectively. Define the infinitesimal generator
\begin{align*}
L_s^{a, x, \mu} \phi &= (b(s, x, \mu, a)\trans, 0_m, 0_n ) D\phi + \frac{1}{2} {\rm Tr}\Big( \left(\begin{matrix}\sigma & \sigma_o\\ I_m & 0_{m \times n} \\ 0_{n \times m} & I_n\end{matrix}\right) \left(\begin{matrix}\sigma\trans & I_{m} & 0_{m \times n}\\ \sigma_o\trans & 0_{n \times m} & I_n\end{matrix}\right) D^2 \phi\Big)\\
& = b(s, x, \mu, a)\trans D_x \phi(x, {\bf w}, {\bf b}) + \frac{1}{2}{\rm Tr}\Big(\big(\sigma\sigma\trans + \sigma_o\sigma_o\trans \big)(s, x, \mu, a)D_{xx} \phi(x, {\bf w}, {\bf b}) + D_{{\bf w} {\bf w}}\phi(x, {\bf w}, {\bf b}) \\
& \;\;\; + D_{{\bf b} {\bf b}} \phi(x, {\bf w}, {\bf b})  + 2\sigma(s, x, \mu, a) D_{x{\bf w}} \phi(x, {\bf w}, {\bf b}) + 2\sigma_o(s, x, \mu, a) D_{x {\bf b}} \phi(x, {\bf w}, {\bf b})\Big).
\end{align*}
For any $\phi \in \Cc_c^\infty(\R^d \times \R^m \times \R^n)$, we define the generator associated with ${\bm \pi}$: $L^{{\bm \pi}, x, \mu}_s \phi(x, {\bf w}, {\bf b}) = \int_{\Ac} L_s^{a, x, \mu} \phi(x, {\bf w}, {\bf b}) {\bm  \pi}(a|s,x, \mu)da$. Heuristically, it is the limit of the infinitesimal generator of the dynamics  \eqref{equ:exploratory_SDE} because the action is sampled from ${\bm \pi}$ independent of $W$ and $B$. 

The rest is to construct a triplet $(X, W, B)$ defined on $(\Omega^e, \Fc^e, \P^e)$ corresponding to the generator $L^{{\bm \pi}, x, \mu} \phi(x, {\bf w}, {\bf b})$, where $(X, W, B)$ satisfy
\begin{align}\label{equ:Xbracket}
\langle d X_s, d X_s\rangle &= \int_{\Ac}\big(\sigma \sigma\trans + \sigma_{o}\sigma_{o}\trans\big)(s, X_s, \mu_s, a){\bm \pi}(a|s, X_s, \mu_s)da,\\
\langle d X_s, d W_s\rangle &= \sigma_{{\bm \pi}}(s,  X_s, \mu_s)ds, \;\langle dX_s, dB_s\rangle =\sigma_{o, {\bm \pi}}(s, X_s, \mu_s)ds,\\
\langle d W_s, dW_s \rangle & = I_{m} ds, \; \langle dB_s, d B_s\rangle = I_{n}ds, \; \langle d  W_s, B_s\rangle = {\bf 0}.\label{WBbracket}
\end{align}
In addition to $B$ and $W$, we add two other $d$-dimensional Brownian motions $\bar B$ and $\bar W$, which are defined on $(\Omega^2, \Fc^2, \P^2)$ and are independent of $B$ and $W$. Recall from \eqref{variance-w.r.t-pi} that ${\rm cov}_{\bm \pi}(\sigma)$ and ${\rm cov}_{\bm \pi}(\sigma_o)$ are positive semidefinite. Hence there exist matrices denoted by ${\rm std}_{\bm \pi}(\sigma)$ and ${\rm std}_{\bm \pi}(\sigma_o)$ such that ${\rm cov}_{\bm \pi}(\sigma)={\rm std}_{\bm \pi}(\sigma){\rm std}_{\bm \pi}(\sigma)\trans$ and ${\rm cov}_{\bm \pi}(\sigma_o)={\rm std}_{\bm \pi}(\sigma_o){\rm std}_{\bm \pi}(\sigma_o)\trans$, and we have
\begin{align}\label{equ:multi_exploratory_HJB}
dX_s^{\bm \pi} & = b_{\bm \pi}(s, X_s^{\bm \pi}, \mu_s^{\bm \pi})ds + \sigma_{{\bm \pi}}(s, X_s^{\bm \pi}, \mu_s^{\bm \pi})d W_s  + \sigma_{o, {\bm \pi}}(s, X_s^{\bm \pi}, \mu_s^{\bm \pi})d  B_s\\
& + {\rm std}_{\bm \pi}(\sigma)(s, X_s^{\bm \pi}, \mu_s^{\bm \pi})d \bar W_s + {\rm std}_{\bm \pi}(\sigma_o)(s, X_s^{\bm \pi}, \mu_s^{\bm \pi}) d \bar B_s. \nonumber
\end{align}
It is readily seen that \eqref{equ:multi_exploratory_HJB} satisfies \eqref{equ:Xbracket}-\eqref{WBbracket}, and hence \eqref{equ:multi_exploratory_HJB} corresponds to the controlled martingale problem with the  generator $L^{{\bm \pi}, x, \mu} \phi(x, {\bf w}, {\bf b})$.

\end{document}